\documentclass[11pt, reqno]{amsart}

 \usepackage{amsmath}
 \usepackage{amssymb}
 \usepackage{bm}

\DeclareMathAccent{\mathring}{\mathalpha}{operators}{"17}

 \usepackage{color}

\newcommand{\mysection}[1]{\section{#1}
      \setcounter{equation}{0}}

\newcommand{\nlimsup}{\operatornamewithlimits{\overline{lim}}}

\newtheorem{theorem}{Theorem}[section]
\newtheorem{lemma}[theorem]{Lemma}

\newtheorem{corollary}[theorem]{Corollary}

\theoremstyle{definition}
\newtheorem{assumption}{Assumption}[section]

\theoremstyle{remark}
\newtheorem{remark}{Remark}[section]

\newcommand\bR{\mathbb{R}}

\newcommand\bS{\mathbb{S}}

\newcommand\bZ{\mathbb{Z}}

\newcommand\cP{\mathcal{P}}

\newcommand\cH{\mathcal{H}}

\newcommand\frH{\mathfrak{H}}

\newcommand\dist{{\rm dist}\,}

\newcommand\esssup{\operatornamewithlimits{ess\,sup\,}}

\begin{document}

\title[Fully nonlinear
parabolic equations]
{An ersatz existence theorem for fully nonlinear
parabolic equations without convexity assumptions}

\author{N.V. Krylov}
\thanks{The  author was partially supported by
NSF Grant DMS-1160569}
\email{krylov@math.umn.edu}
\address{127 Vincent Hall, University of Minnesota,
 Minneapolis, MN, 55455}

\keywords{Fully nonlinear equations, parabolic equations,
Bellman's equations, finite differences}

 \subjclass[2010]{35K55,39A12}

\begin{abstract}
We show that for any uniformly parabolic fully
nonlinear second-order equation with bounded measurable
``coefficients'' and bounded ``free'' term 
 in the whole space or in any
 cylindrical smooth domain  with  smooth
boundary data
one can find an
approximating equation which has a   continuous  solution
with the first  and the second spatial derivatives
under control: bounded in the case of the whole space
and locally bounded in case of equations in cylinders.
 The approximating equation is constructed in such a
way that it modifies the original one only for large values of the
second spatial derivatives of the unknown  function.
This is different from a previous work of Hongjie Dong
and the author where the modification was done
for large values of the
unknown function and its spatial derivatives.
\end{abstract}

\maketitle

\mysection{Introduction}
                                       \label{section 9.14.1}

This article is a natural continuation of
\cite{DK} and is written in the same framework.
We are given a function  $H(u,t,x)$,
$$
 u=(u',u''),\quad
u'=(u'_{0},u'_{1},...,u'_{d}) \in\bR^{d+1},\quad u''\in\bS,\quad
(t,x)\in\bR^{d+1}, 
$$ 
where $\bS$ is the
set of symmetric $d\times d$ matrices, and
we are dealing with some modifications of
the parabolic  equation 
\begin{equation}
                                                \label{7.29.1}
\partial_t v(t,x)+H[v](t,x):= \partial_t v(t,x)+H(v(t, x),D v(t,x),D^{2}v(t, x),t, x)=0
\end{equation}
in subdomains of $(0,T)\times \bR^d $, where $T\in(0,\infty)$,
$$
\bR^{d}=\{x=(x_{1},...,x_{d}):x_{1},...,x_{d}\in \bR\},
$$

$$
\partial_t=\frac{\partial}{\partial t},\quad
 D^{2}u=(D_{ij}u),\quad Du=(D_{i}u),\quad
D_{i}=\frac{\partial}{\partial x_{i}},
\quad D_{ij}=D_{i}D_{j}.
$$
As in \cite{DK} we are looking for a uniformly elliptic
 operator $P[v]$
given by a convex positive-homogeneous of degree one
function $P$ independent of $(t,x)$ such that
the boundary-value problem we are interested in for the equation 
\begin{equation}
                                               \label{9.23.2}
\partial_t v+\max(H[v],P[v]-K)=0 
\end{equation}
would be solvable in the classical sense (a.e.) for any constant $K
{\color{black}>}0$.
However, unlike \cite{DK} we do not allow $P[v]$ to depend
on $v$ and its first derivatives, so that $P(u,t,x)=P(u'')$.
 A big advantage of this approach is that we do not need
Lipschitz continuity of $H(u,t,x)$ with respect to $u'$
but rather not faster than linear growth
of $H(u',0,t,x)$ as $|u'|\to\infty$.   
Actually, our results even in the particular case of $H$
{\em independent\/} of $u'$ 
  play a major role in  paper
\cite{Kr13}  aimed at  proving
that $L_{p}$-viscosity solutions of \eqref{7.29.1}
are in $C^{1+\alpha}$ provided that ``the main coefficients''
of $H$ are in VMO.

 Solvability theory for uniformly nondegenerate
 parabolic equations like 
\eqref{7.29.1} and its elliptic counterparts in  H\"older
classes of functions is   well developed in case $H$ is 
convex or concave
in $u''$ (see, for instance, \cite{GT}, \cite{Kr85},
\cite{Li}). In case this condition is abandoned
 N. Nadirashvili and S. Vl\v{a}dut \cite{NV} gave an example
of elliptic fully nonlinear equation which does not admit
classical (or even $C^{1+\alpha}$ viscosity) solution.
For that reason the interest in Sobolev space theory
became even more justifiable. In \cite{Kr2} the author proved the first
existence (and uniqueness) result for fully nonlinear
elliptic equations under relaxed convexity assumption
for equations with VMO ``coefficients''. Previously,
M. G. Crandall, M. Kocan,
and A. \'Swi{\c e}ch \cite{CKS00}  established the {\em
solvability\/} in local Sobolev spaces of 
the boundary-value problems
for fully nonlinear parabolic equations 
 and    N. Winter \cite{Wi09}
  established the solvability in the global $W^2_p$-space of the
associated boundary-value problem in the elliptic case.
In the solvability parts of these two papers
 $H$ is assumed to be convex in $u''$
and, basically, have continuous  ``coefficients'' (actually, it is 
assumed to be uniformly sufficiently close to the ones
having continuous  ``coefficients'').

There is also a  quite extensive a priori estimates side of the story
(not involving the solvability)
for which we refer the reader to \cite{CKS00}, \cite{DK},
\cite{Wi09} and the references therein.

Apart from Theorems \ref{theorem 9.12.1} 
about the solvability of equations in the whole space
and 
\ref{theorem 9.23.01} about that in cylinders, which are proved in Sections
\ref{section 9.21.2} and \ref{section 9.21.1}, respectively,
Theorem \ref{theorem 9.15.1} proved in Section \ref{section 9.21.3}
is also one of our main results. Roughly 
speaking, it says that as $K\to\infty$ the solutions of
\eqref{9.23.2} converge to the maximal $L_{d+1}$-viscosity
solution of \eqref{7.29.1}. The existence of
the maximal $L_{p}$-viscosity
solution for elliptic case was proved in \cite{JS05}.
We provide a method which in principle allows one to find it.

Finally, Section \ref{section 9.14.2} contains
our main results, Section \ref{section 12.13.2}
is devoted to reducing Theorem \ref{theorem 9.12.1}
to a simpler statement, {\color{black}
and in the rather long Section \ref{section 5.3.1}
we prepare necessary tools in order to be able to prove 
our main results by using finite-difference approximations.
}

\mysection{Main results}
                                       \label{section 9.14.2}
 
Fix some constants $\delta\in(0,1)$ and $K_{0}\in[0,\infty)$.
 Set
$$
\bS_{\delta}=\{a\in\bS:\delta|\xi|^{2}\leq a_{ij}\xi_{i}\xi_{j}
\leq\delta^{-1}|\xi|^{2},\quad \forall\, \xi\in\bR^{d}\}, 
$$
 where and
everywhere in the article the summation convention is enforced.

\begin{assumption}
                                    \label{assumption 9.23.1}
(i) 
The function   $H(u,t,x)$
is measurable with respect to $(u',t,x)$ for any $u''$, 
  Lipschitz continuous in $u''$, and at all points of differentiability
of $H$ with respect to $u''$ we have
$H_{u''}\in \bS_{\delta}$,

(ii) The number
$$
\bar{H} :=\sup_{u',t,x}\big(|H (u',0,t,x)|-K_{0}|u'|\big)
\quad(\geq0)
$$
is finite,

(iii) There is an increasing
 continuous function $\omega(r)$, $r\geq0$,
such that $\omega(0)=0$ and
$$
|H(u',u'',t,x)-H(v',u'',t,x)|\leq\omega(|u'-v'|)
$$
for all $u,v,t$, and $x$.

\end{assumption}

 By Theorem  3.1 of \cite{Kr11} there exists
a set 
\begin{equation}
                                              \label{6.1.2}
\Lambda=\{ l_{1},...,l_{m}\}
\subset \bZ^{d },
\end{equation}
$m=m(\delta,d)\geq d$, chosen on the sole
 basis of
knowing $\delta$ and $d$ and 
 there exists 
a constant  
$$
\hat{\delta}=\hat{\delta}(\delta,d ) 
\in(0,\delta/4]
$$
such that:

(a) We have $l_{i}=e_{i}$ and
$$
 e_{i}\pm e_{j}\in \{l_{1},...,l_{m}\}
=\{-l_{1},...,-l_{m}\}
$$
for all $i,j=1,...,d$, where $e_{1},...,e_{d}$ is the 
 standard orthonormal basis of $\bR^{d}$;

(b) There   exist
real-analytic functions $\lambda_{1}(a),...,
\lambda_{m}(a)$ on $\bS_{\delta/4}$ such that
for any $a\in\bS_{\delta/4}$
\begin{equation}
                                         \label{4.8.01}
a=\sum_{k=1}^{m}\lambda_{k}(a)l_{k}l_{k}^{*},
\quad \hat{\delta}^{-1}\geq\lambda_{k}(a)\geq\hat{\delta}
,\quad \forall k.
\end{equation}

Introduce
\begin{equation}
                                                       \label{6.1.1}
\cP(z'')  =\max_{\substack{\hat{\delta}/2\leq a_{k}\leq
2\hat{\delta}^{-1} \\k=1,...,m} }
 \sum_{k=1}^{m} a_{k}
z''_{k} ,
\end{equation}
and for $ u'' \in \bS$ define
$$
P( u'')=\cP( \langle u''l_{1},l_{1}\rangle,...,
\langle u''l_{m},l_{m}\rangle),
$$
where $\langle\cdot,\cdot\rangle$ is the scalar product in $\bR^{d}$.
Naturally, by $P[v]$ we mean a differential operator
constructed as in \eqref{7.29.1}.

{\color{black}
\begin{remark}
                                             \label{remark 6.5.1}
Obviously $P(u'')$ is Lipschitz continuous and at each point of its
differentiability we have
$$
P_{u''_{ij}}\xi_{i}\xi_{j}\geq
\min_{\substack{\hat{\delta}/2\leq a_{k}\leq
2\hat{\delta}^{-1} \\k=1,...,m} }
 \sum_{k=1}^{m}a_{k}\langle \xi,l_{k}\rangle^{2}\geq
(\hat{\delta}/2)\sum_{k=1}^{d}\langle \xi,l_{k}\rangle^{2}
=(\hat{\delta}/2)|\xi|^{2}.
$$
It follows that there exists
  $\check{\delta}\in(0,1)$ depending only
on $\delta,d$, such that the function $F_{K}=\max(H,P-K)$
is Lipschitz continuous with respect to $u''$ and at each point of its
differentiability with respect to $u''$ we have 
$F_{Ku''}\in\bS_{\check{\delta}}$.
\end{remark}

}

Let $\Omega$ be an open bounded subset of $\bR^{d}$ with
$C^{2}$ boundary. We denote the parabolic boundary of the cylinder
$\Omega_T=(0,T)\times \Omega$ by
$$
\partial'\Omega_{T}
=(\partial \Omega_{T})\setminus (\{0\}\times \Omega).
$$

Below, for $\alpha\in(0,1)$,
 the parabolic spaces $C^{\alpha/2,\alpha}$ and elliptic 
spaces $C^{\alpha}$ are usual H\"older spaces.
These  spaces are provided with  natural norms.
 
\begin{theorem}
                                          \label{theorem 9.12.1}

Let  $K{\color{black}>}0$ be a fixed constant,  
$g\in W^{1,2}_{\infty}(\Omega_{T})\cap C(\bar \Omega_{T})
$. Suppose that Assumption \ref{assumption 9.23.1}
is satisfied. Then
equation \eqref{9.23.2} in $\Omega_{T}$  
with  boundary
condition $v=g$ on $\partial'\Omega_{ T}$ has a 
  solution $v\in
C(\bar{\Omega}_T)\cap W^{1,2}_{\infty,\text{loc}}(\Omega_{T})$.
In addition, for all
$i,j$, and $p\in(d+1,\infty)$,
\begin{equation}
                                             \label{9.22.3}
|v|, |D_{i}v|,\rho |D_{ij} v |,|\partial_t v |\leq N(\bar{H}+K
+\|g\|_{ W^{1,2}_{\infty}(\Omega_{T})})\quad\text{in} \quad \Omega_T \quad (a.e.),
\end{equation}
\begin{equation}
                                                \label{1.13.2}
\|v\|_{W^{1,2}_{p}(\Omega_T)}\leq
N_{p}(\bar{H}+K+\|g\|_{W^{1,2}_{p}(\Omega_T)}), \end{equation}
\begin{equation}
                                                \label{2.28.1}
\|v\|_{ C^{\alpha/2,\alpha}(\Omega_T)}\leq N (
\bar{H}+\|g
\|_{ C^{\alpha/2,\alpha}(\Omega_T)}),
\end{equation}
where
$$
\rho=\rho(x)=\dist(x,\bR^{d}\setminus\Omega),
$$
$\alpha\in(0,1)$ is a constant depending only on $d$ and $\delta$,
   $N$ is a constant depending only on $\Omega$,
$T$, $K_{0}$, and $\delta$, 
whereas $N_{p}$ only depends on $\Omega$,  
$T$, $K_{0}$, $\delta$, and
$p$ (in particular, $N$ and $N_{p}$ are independent
of $\omega$).

\end{theorem}

Here is our second main result.
 
\begin{theorem}
                                    \label{theorem 9.23.01}
Let  $K{\color{black}>}0$ be a fixed constant, and
$g\in  W ^{ 2}_{\infty} (\bR^{d})$.
Then  equation \eqref{9.23.2}
in $Q_{T}:=(0,T)\times\bR^{d}$ (a.e.) with terminal 
condition $v(T,x)=g(x)$  has a  solution $v\in
W^{1,2}_{\infty}(Q_{T})\cap C(\bar Q_{T}) $.
In addition, 
$$
|v|, |Dv|, |D^{2}v |,|\partial_t v |\leq N(\bar{H}+K
+\|g\|_{ W ^{ 2}_{\infty} (\bR^{d})})\quad\text{in} \quad Q_T \quad (a.e.),
$$
$$
\|v\|_{ C^{\alpha/2,\alpha}(Q_T)}\leq N 
(\bar{H}+\|g
\|_{ C^{ \alpha}(\bR^{d})}),
$$
where
$\alpha\in(0,1)$ is a constant depending only on $d$ and $\delta$ 
 and  $N$ is a constant depending only on $T$, $K_{0}$, $d$,
 and $\delta$.

\end{theorem}

Before stating our third main result 
introduce the following.
\begin{assumption}
                                         \label{assumption 9.24.1}
The function $H$ is a nonincreasing function of $u'_{0}$, which
is continuous with respect to $u'_{0}$ uniformly with respect
to other variables,
and is Lipschitz continuous with respect to $(u'_1,...,u'_{d})$
with constant independent of other variables.  
\end{assumption}
{\color{black}
We also remind the reader a definition from \cite{CKS00}
according to which   we say that a function
$u(t,x)$ is an $L_{p}$-viscosity subsolution of \eqref{7.29.1}
in $\Omega_{T}$
if for any $(t_{0},x_{0})\in \Omega_{T}$ and any 
$\phi\in W^{1,2}_{p,loc}
(\Omega_{T})$ for which   $u-\phi$ is continuous at $(t_{0},x_{0})$
and attains a local
maximum at $(t_{0},x_{0})$, we have
$$
 \lim _{ r\downarrow0}\esssup_{C_{r}(t_{0},x_{0}) }
\big[\partial_{t}\phi(t,x) +
H(u(t,x),D\phi(t,x),D^{2}\phi(t,x),t,x)\big]\geq 0,
$$
where
$$
C_{r}(t_{0},x_{0})=(t_{0},t_{0}+r^{2})\times\{x\in\bR^{d}:
|x-x_{0}|<r\}.
$$
In a natural way one defines $L_{p}$-viscosity supersolution
and calls a function an $L_{p}$-viscosity solution if it is an
$L_{p}$-viscosity supersolution and an $L_{p}$-viscosity subsolution. 
} The reader is referred 
to \cite{CKS00} for numerous properties of $L_{p}$-viscosity
solutions.

Observe that under Assumption \ref{assumption 9.24.1}
 the solutions $v=v_{K}$ constructed in Theorem 
\ref{theorem 9.12.1}   for each $K$
 are unique
and decrease as $K\to\infty$.

\begin{theorem}
                                          \label{theorem 9.15.1}
 Suppose that Assumptions \ref{assumption 9.23.1}
and \ref{assumption 9.24.1} are
  satisfied. Then, 
as $K\to\infty$, $v_{K}$ converges uniformly on
$\bar\Omega_{T}$ to a continuous function $v$ which
is an $L_{d+1}$-viscosity solutions of \eqref{9.23.2} in $\Omega_{T}$  
with  boundary
condition $v=g$ on $\partial'\Omega_{ T}$.
Furthermore,  $v$ is 
the maximal $L_{d+1}$-viscosity subsolution
of class $C(\bar{\Omega}_{T})$ of this problem.
\end{theorem}

\mysection{Reduction of Theorem \ref{theorem 9.12.1}
to a simpler statement}

Denote by $C^{1,2}(\bar\Omega_{T})$
the set of functions $g(t,x)$
such that $g,Dg,D^{2}g,\partial_{t}g\in
C(\bar\Omega_{T})$. The norm in
$C^{1,2}(\bar\Omega_{T})$ is introduced in an obvious way.
                                           \label{section 12.13.2}

\begin{lemma}
                                            \label{lemma 9.22.1}
Suppose that the assertions of Theorem \ref{theorem 9.12.1}
hold true if $g\in C^{1,2}(\bar\Omega_{T})$ and,
 in addition to Assumption
\ref{assumption 9.23.1}, for any
$s,t\in \bR$, $x,y\in\bR^{d}$,  $u=(u',u'')$, 
and $v=(v',v'')$,
\begin{equation}
                                    \label{9.31.2}
|H(u,t,x)-H(u,s,y)|\leq N'
(|t-s|+|x-y|)(1+|u|),
\end{equation}
\begin{equation}
                                           \label{9.22.1}
  |H(u',u'',t,x)-H(v',u'',t,x)|\leq N'
|u'-v'|
\end{equation}
where $N'$ is
independent of $t,s,x,y,u$, and $v$.
Then the assertions of Theorem \ref{theorem 9.12.1}
hold true without these additional assumptions as well.

\end{lemma}

\begin{proof} First we suppose that the assertions of Theorem \ref{theorem 9.12.1}
hold true with $g$ as there  but under
the additional assumption that \eqref{9.31.2} and 
\eqref{9.22.1} hold.

Note that
\begin{equation}
                                              \label{11.7.1}
|H(u,t,x)|\leq |H(u,t,x)-H(u',0,t,x)|+K_{0}|u'|+\bar{H}
\leq \bar{H}+N(K_{0},d,\delta)|u|.
\end{equation}
Then
let $B_{1}$ be the open unit ball in $\bR^{d+1}$
centered at the origin. Take a nonnegative
$\zeta\in C^{\infty}_{0}(B_{1})$, which integrates
to one 
and introduce
$H_{n}(u,t,x)$ as the convolution of 
$H(u,t,x)$ and $n^{ d+1}\zeta(  nt,nx)$  performed
with respect to $( t,x)$. 
Observe that   $H_{n}$  satisfies
 Assumption \ref{assumption 9.23.1}
with the same constant $\delta $, whereas
$$
|H_{n}(u,t,x)-H_{n}(u,s,y)|\leq  n|B_{1}| (|t-s|+|x-y|)
\sup_{z}|H(u,z)|
\sup| D\zeta|,
$$
where $|B_{1}|$ is the volume of $B_{1}$,
and \eqref{9.31.2}
(with $N'$ of course depending on $n$)
 is satisfied due to \eqref{11.7.1}.

Next, define $H^{n}(u,t,x)$  as the convolution of 
$H_{n}(u,x)$ and $n^{ d+1}\zeta(n u', t, x)$  performed
with respect to $u'$. Obviously, for each $n$, $H^{n}$
satisfies \eqref{9.31.2} with a constant $N'$.
Furthermore, for any $k=0,...,d$
$$
H^{n}_{ u'_{k}}(u',u'',t,x)=n\int_{\bR^{ d+1}}
 H_{n}(u'-v'/n,u'',t ,x ) \zeta_{u'_{k}}(v' )\,dv'
$$
$$
=n\int_{\bR^{ d+1}}
[H_{n}(u'-v'/n,u'',t ,x )-H_{n}(u',u'',t,x)]
 \zeta_{u'_{k}}(v' )\,dv'.
$$
It follows that
$$
|H^{n}_{ u'_{k}}(u ,t,x)|\leq n\omega(1/n)|B_{1}|\sup| D\zeta|,
$$
so that $H^{n}$ also satisfies \eqref{9.22.1}.

Now by assumption there exist solutions 
$v^{n} \in  C(\bar{\Omega}_T)\cap 
W^{1,2}_{\infty,\text{loc}}(\Omega_{T})$
of
\begin{equation}
                                       \label{3.29.2}
\partial_{t}v^{n}+\max(H^{n}[v^{n}],P[v^{n}]-K)=0
\end{equation}
in $\Omega_{T}$ (a.e.) with  boundary condition $v^{n}=g$,
for which estimates \eqref{9.22.3},
\eqref{1.13.2}, and \eqref{2.28.1} hold with
$v^{n}$ in place of $v$ with the constants $N$ and $N_{p}$
from Theorem \ref{theorem 9.12.1} and with  
$$
\bar{H}^{n}=\sup_{u',t,x} \big(|H^{n}(u',0,t,x)|-K_{0}|u'|\big)
\quad\quad(\leq\bar{H}+K_{0}n^{-1})
$$
in place of $\bar{H}$. 
 Furthermore, being uniformly bounded and uniformly continuous,
the sequence $\{v^{n}\}$ has a subsequence uniformly converging
to a function $v$, for which 
\eqref{9.22.3},
\eqref{1.13.2}, and \eqref{2.28.1}, 
of course, hold 
and $v  \in 
 C(\bar{\Omega}_T)\cap 
W^{1,2}_{\infty,\text{loc}}(\Omega_{T})$. For simplicity 
of notation we suppose that the whole sequence $v^{n}$
converges.

Observe that 
\begin{equation}
                                              \label{10.3.4}
\partial_{t}v^{m}+\check{H}^{n}_{K}[v^{m}]\geq0
\end{equation}
in $\Omega_{T}$ (a.e.) for all $m\geq n$, where
$$
\check{H}^{n}_{K}(u,t,x):=\sup_{k\geq n}
\max(H^{k}(v^{k}(t,x),Dv^{k}(t,x),u'',t,x),P(u'')-K).
$$

In light of \eqref{10.3.4} and the fact that
 the norms $\|v^{n}\|_{W^{1,2}_{p}(\Omega_{T})}$ are bounded,
by Theorem  3.5.9  of \cite{Kr85}
 we have 
\begin{equation}
                                              \label{10.3.2}
\partial_{t}v+\check{H}^{n}_{K}[v ]\geq0 
\end{equation}
in $\Omega_{T}$ (a.e.).

Now we notice that by embedding theorems
$Dv^{k}$ are locally uniformly continuous in $\Omega_{T}$
and this and the convergence $v^{n}\to v$ implies 
 by a standard fact of calculus 
that $Dv^{k}$ converge to $Dv$ locally uniformly
in $\Omega_{T}$. Also
$$
|H^{k}(u,t,x)-H_{k}(u,t,x)|\leq\omega(1/k),
$$
$$
|H_{k} (v^{k}(t,x),Dv^{k}(t,x),D^{2}v(t,x),t,x)
-H_{k} (v (t,x),Dv (t,x),D^{2}v(t,x),t,x)|
$$
$$
\leq\omega\big(|v^{k}-v|(t,x)+|Dv^{k}-Dv|(t,x)\big),
$$
which along with what was said above implies that
\begin{equation}
                                              \label{9.22.5}
\partial_{t}v+\hat{H}^{n}_{K}[v ]\geq-\varepsilon_{n}
\end{equation}
in $\Omega_{T}$ (a.e.), where the functions $\varepsilon_{n}
\to 0$ in $\Omega_{T}$ (even locally uniformly) and
$$
\hat{H}^{n}_{K}(u,t,x):=\sup_{k\geq n}
\max(H_{k}(u,t,x),P(u'')-K).
$$

  Then we
  notice that by the Lebesgue differentiation theorem
for any $u$
\begin{equation}
                                              \label{10.3.1}
\lim_{n\to\infty}\hat{H}^{n}_{K}(u,t,x)=\max(H(u,t,x),P(u)-K)
\end{equation}
for almost all $(t,x)$. Since for any bounded set
$\Gamma$ in the range of $u$, $\hat{H}^{n}_{K}(u,t,x)$ are 
uniformly
continuous on $\Gamma$ uniformly with respect to $(t,x)$ and $n$,
there exists a subset of $\Omega_{T}$ of full measure such that
\eqref{10.3.1} holds on this subset for all $u$.

We conclude that in $\Omega_{T}$ (a.e.)
\begin{equation}
                                              \label{10.3.3}
\partial_{t}v+\max(H [v],P[v]-K)\geq0.
\end{equation}

The opposite inequality is obtained by considering
$$
 \inf_{k\geq n}
\max(H^{k}(v^{k}(t,x),Dv^{k}(t,x),u'',t,x),P(u'')-K).
$$

The fact that it suffices to prove Theorem
\ref{theorem 9.12.1} under the additional
assumption that $g\in C^{1,2}(\bar\Omega_{T})$
is proved by mollifying $g$
and using a very simplified version of the above
arguments.
The lemma is proved.
\end{proof}

 Next,
 we show that one may assume that $H$ is
boundedly inhomogeneous with respect to $u''$
{\color{black}(in the sense described in
Lemma \ref{lemma 9.29.2} below)}. Introduce 
$$
P_{0}(u)=P_{0}(u'')=\max_{a\in\bS_{\delta/2}}
 a_{ij}u''_{ij} , 
$$ 
where  the
summation    is performed
 before the maximum is taken.
It is easy to see that $P_{0}[u]$ is     Pucci's operator:
$$
P_{0}(u)=-(\delta/2)\sum_{k=1}^{d}\lambda_{k}^{-}(u'')
+2\delta^{-1}\sum_{k=1}^{d}\lambda_{k}^{+}(u''),
$$
where
$\lambda_{1}(u''),...,\lambda_{d}(u'')$ are the eigenvalues of $u''$ and
$a^{\pm}=(1/2)(|a|\pm a)$.

Observe that
$$
 P(u)  =\max_{\substack{\hat{\delta}/2\leq a_{k}\leq
2\hat{\delta}^{-1} \\k=1,...,m} }   \sum_{i,j=1}^{d} \sum_{k=1}^{m}
a_{k}l_{ki}l_{kj}u''_{ij}.
$$
Moreover,
 owing to property (b) in   Section
\ref{section 9.14.2}, the collection of matrices
$$
\sum_{k=1}^{m}a_{k}l_{k}l_{k}^{*}
$$ such that $\hat{\delta}\leq
a_{k}\leq \hat{\delta}^{-1},k=1,...,m$, covers   $\bS_{\delta/4}$.
Hence,
$$
P(u)\geq -(\delta/4)\sum_{k=1}^{d}\lambda_{k}^{-}(u'')
+4\delta^{-1}\sum_{k=1}^{d}\lambda_{k}^{+}(u'')
$$
\begin{equation}
                                          \label{3.6.1}
\geq P_{0}(u)+(\delta/4)\sum_{k=1}^{d}|\lambda_{k} (u'')|.
\end{equation}

In particular, $P_{0}\leq P$ and therefore,
$$
\max(H,P-K)=\max(H_{K},P-K),
$$
where $H_{K}=\max(H,P_{0}-K)$. It is
easy to see that the function $H_{K}$ satisfies Assumption
 \ref{assumption 9.23.1} (i)
 with $\delta/2$  in place of $\delta$, satisfies 
Assumption
 \ref{assumption 9.23.1} (iii) with the same function $\omega$,
and 
$$
|H_{K}(u',0, t,x) |\leq
|H (u',0,t,x)|
\leq K_{0}|u' |+\bar{H},
$$
so that the number
$$
\bar{H}_{K}:=\sup_{u',t,x}\big(|H_{K}(u',0,t,x)|-K_{0}|u'|\big)
\quad(\leq \bar{H})
$$
is finite and
Assumption
 \ref{assumption 9.23.1} (ii) is also satisfied.
Also observe that $H_{K}$
 satisfies \eqref{9.31.2} and \eqref{9.22.1}
with the same constant $N'$.

To continue we note the following.
\begin{lemma}
                                          \label{lemma 9.29.2}
There is a constant $\kappa>0$ depending only on $\delta$ and $d$ 
such
that 
\begin{equation}
                                            \label{9.29.2}
H \leq  P_{0}-\kappa
 |u'' | +K_{0} |u' | +
\bar{H}_{(+)},
\end{equation}
 \begin{equation}
                                            \label{3.6.3}
H_{K} \leq  P -\kappa
  |u'' | +K_{0} |u' | +
\bar{H}_{(+)},
\end{equation}
where
$$
\bar{H}_{(+)}:=
\sup_{u',t,x}\big( H^{+}(u',0,t,x) -K_{0}|u'|\big)\leq \bar{H}_{K}.
$$
Furthermore, $H_{K}$ is boundedly inhomogeneous with respect to 
$u''$ in
the sense that at all points of differentiability 
of $H_{K}(u,t,x)$ with
respect to $u''$ 
\begin{equation}
                                              \label{9.30.01}
|H_{K}(u,t,x)-H_{Ku''_{ij}}(u,t,x)u''_{ij}  |\leq
N(K_{0}+1)(\bar{H}_{K} +K +|u'|),
\end{equation} 
where $N$ depends only on  $d$ and $\delta$.

\end{lemma}

\begin{proof} {\color{black} To prove
  \eqref{9.29.2}  
  fix $u',t,x$ and denote by $D_{H}$ the set
in $\bS$ of points at which $H(u',u'',t,x)$ is differentiable
with respect to $u''$. Since $H$ is Lipschitz continuous
with respect to $u''$, by Rademacher's theorem, the set
$D_{H} $ has full measure. By Fubini's theorem the sets of
full measure contain almost entirely almost any ray, so that
for almost any $u''$ the set of $s\in[0,1]$ such that
$su''\in D_{H} $ also has a full measure.
Furthermore, since the function $H(u',su'',t,x)$
is a Lipschitz continuous function of $s$, it is absolutely continuous
and (Hadamard's formula)
\begin{equation}
                                                       \label{5.30.1}
H(u,t,x)=H(u',0,t,x)+\int_{0}^{1}\frac{\partial}{\partial s}
H(u',su'',t,x)\,ds.
\end{equation}
At points $s\in[0,1]$ such that $su''\in D_{H}$ the function
$H(u',su''+v'',t,x)$ is differentiable with respect to $v''$ at
$v''=0$. Hence by calculus at those points $s$ which have full measure
we obtain 
$$
\frac{\partial}{\partial s}
H(u',su'',t,x)=u''_{ij}H_{u''_{ij}}(u',su'',t,x),
$$
which along with \eqref{5.30.1} and the assumption that $H_{u''}
\in\bS_{\delta}$ shows that for   almost all $u''$
\begin{equation}
                                                  \label{5.30.4}
H(u,t,x)=H(u',0,t,x)+a_{ij}u''_{ij},
\end{equation}
where $a=(a_{ij})\in\bS_{\delta}$ is defined by
$$
a=\int_{0}^{1}
H_{u'' }(u',su'',t,x)\,ds.
$$
We have proved that for almost any $u''$ there exists an
$a\in\bS_{\delta}$ such that \eqref{5.30.4} holds. Since $H$
is continuous with respect to $u''$ and $\bS_{\delta}$
is a compact set, \eqref{5.30.4} holds for any $u''$
with an appropriate $a\in\bS_{\delta}$.
We basically repeated  part of
the proof of Lemma 2.2 in \cite{Kr12.2}.

It follows that
$$
H(u,t,x)\leq (H^{+}(u',0,t,x)-K_{0}|u'|)+K_{0}|u'|
+\sup_{a\in\bS_{\delta}}a_{ij}u''_{ij}.
$$

Here the first term on the right is less than $\bar{H}_{+}$
by definition and the last term equals
$$
- \delta \sum_{k=1}^{d}\lambda_{k}^{-}(u'')
+ \delta^{-1}\sum_{k=1}^{d}\lambda_{k}^{+}(u'')
=P_{0}(u)-(\delta/2) \sum_{k=1}^{d}\lambda_{k}^{-}(u'')
-\delta^{-1}\sum_{k=1}^{d}\lambda_{k}^{+}(u'')
$$
$$
\leq P_{0}(u)-(\delta/2)\sum_{k=1}^{d}|\lambda_{k} (u'')|.
$$
This certainly implies \eqref{9.29.2}.

}

Estimate \eqref{3.6.3} now also follows since $P_{0}\leq P$.
To prove \eqref{9.30.01} note that if
\begin{equation}
                                            \label{9.30.2}
\kappa
 |u'' | 
 \geq   K_{0} |u' |+\bar{H}_{(+)}+K,
\end{equation} 
then by \eqref{9.29.2}
$$
H (u,t,x)\leq P_{0}(u)-\kappa
  |u'' | +K_{0} |u' | +
 \bar{H}_{(+)}\leq P_{0}(u)-K,
$$ 
so that $H_{K}(u,t,x)=P_{0}(u)-K$  and the left-hand side of
\eqref{9.30.01} is just $K$ owing to  the  fact that $P_{0}$ is
positive homogeneous of degree one. On the other hand, if the
opposite inequality holds in \eqref{9.30.2}, then 
it follows from
$$
|H_{K}(u,t,x)|\leq
|H_{K}(u,t,x)-H_{K}(u',0,t,x)|+|H_{K}(u',0,t,x)|
$$
$$
\leq N|u''|+K_{0}|u'|+\bar{H}_{K}\leq
N(K_{0}+1)(|u'|+K +\bar{H}_{(+)}+\bar{H}_{K})
$$
that the left-hand side of \eqref{9.30.01} is dominated by
$$
N(K_{0}+1)(|u'|+K+\bar{H}_{(+)}+\bar{H}_{K}).
$$
After that it only remains to notice that
$$
H(u',0,t,x)\leq\max(H(u',0,t,x),-K)=H_{K}(u',0,t,x),
$$
$$
H^{+}(u',0,t,x)\leq |H_{K}(u',0,t,x)|,\quad \bar{H}_{(+)}
\leq\bar{H}_{K}.
$$
The lemma is proved.
\end{proof}

This lemma
shows that in the rest of the proof of Theorem
\ref{theorem 9.12.1}
 we may   assume that
not only Assumption \ref{assumption 9.23.1} 
is satisfied with $\delta/2$ in place of $\delta$
and \eqref{9.31.2} and \eqref{9.22.1}
 hold  with a constant $N'$, but also at all 
points of
differentiability of $H$ with respect to $u$ 
\begin{equation}
                                              \label{9.10.2}
|H  (u,t,x)-H_{ u''_{ij}}(u,t,x)u''_{ij}  |\leq
  K'_{0}(\bar{H} +K +|u'|),
\end{equation} 
where $K'_{0}=N(\delta,d)(K_{0}+1)$  and
 \begin{equation}   
                                            \label{3.6.4}
H  \leq  P -\kappa
  |u'' | +K_{0} |u' | +
\bar{H},
\end{equation}
where $\kappa$ is the constant from Lemma \ref{lemma 9.29.2}.

As a result of the above arguments we see that to prove Theorem
\ref{theorem 9.12.1} it suffices to prove the following.
\begin{theorem}
                                   \label{theorem 10.5.1}
Suppose that $g\in C^{1,2}(\bar\Omega_{T})$ and
Assumption \ref{assumption 9.23.1} is satisfied with
$\delta/2$ in place of $\delta$. Also assume that 
\eqref{9.10.2}  
holds  at all points of differentiability of  
$H (u,t,x)$ with respect to
$u$.
 Finally, assume that
estimates \eqref{9.31.2} and \eqref{9.22.1}
with a constant $N'$ and
\eqref{3.6.4}  
   hold
 for any $t,s\in \bR$, $x,y\in\bR^{d}$, and $u,v$.
Then the assertions of Theorem \ref{theorem 9.12.1}  hold true.
\end{theorem}

{\color{black}
\mysection{Some auxiliary results}
                                       \label{section 5.3.1}
In this section the assumptions of Theorem \ref{theorem 10.5.1}
are supposed to be satisfied.
First we show that one can rewrite $H[v]$ in such a way
that only pure second order derivatives along $l_{k}$'s
 of $v$ enter ($l_{k}$'s are introduced in connection with
\eqref{6.1.2}).
Recall that $K'_{0}$ is introduced after \eqref{9.10.2},
  define
$$
I=[-K'_{0} ,K'_{0}],\quad J=[-2K'_{0} ,2K'_{0}],
\quad C''=I\times\bS_{\delta/2},\quad
B''=J\times
\bS_{\delta/4},
$$
and also recall that $H_{u''}\in\bS_{\delta/2}$
at all points of differentiability of $H$ with respect to $u''$.
In terminology of \cite{Kr11} this means that for any $(t,x)$
$$
H(\cdot,t,x)\in\frH_{C''}\subset \frH_{B''} .
$$

Next,  for $u',(t,x)\in\bR^{d+1}$,  and $y''\in\bS$ 
introduce   
$$
B(u',y'',t,x)=\{(f,l'')\in B'':(\bar{H} +K +|u'|)
f+l_{ij}y''_{ij}\leq H(u',y'',t,x)\}.
$$
As follows from \cite{Kr11} or from the properties of $H$, the sets
$B(u',y'',t,x)$ are closed and nonempty.
We now  recall  \eqref{4.8.01} and for $u', (t,x)\in\bR^{d+1}$,
and $z''\in \bR^{m}$  ($m$ is the same as in  \eqref{4.8.01}
and $z''$ in this section is a vector rather than a matrix) define
$$
\cH(u',z'',t,x)
=\inf_{y''\in\bS}\max_{(f,l'')\in
B(u',y'',t,x)}\big[(\bar{H} +K +|u'|)
f+\lambda_{k}(l'')z_{k}''\big].
$$}

By Theorem 5.2 and Corollary 5.3 of \cite{Kr11}
{\color{black}(modified in an obvious way by replacing $1+|u'|$
with $\bar{H} +K +|u'| $)},
the function
$\cH$ is  measurable, Lipschitz continuous
with respect to $ z'' $ with constant
 independent 
of $(u',t,x)$,
\begin{equation}
                                    \label{9.8.1}
H(u,t,x)=\cH (u',\langle u''l_{1},l_{1}\rangle,
...,\langle u''l_{m},l_{m}\rangle,t,x)
\end{equation}
for all values of arguments, 
{\color{black} where $l_{k}$ are taken from
\eqref{6.1.2}}, and at all points
of differentiability of $\cH $ with respect to $z''$ we have  
\begin{equation}
                                                \label{9.8.2}
D_{z''}\cH (u',z'',t,x)
\in [ \hat\delta,\hat\delta ^{-1}]^{m},
\end{equation}
\begin{equation}
                                      \label{9.8.3}
 {\color{black}(\bar{H} +K +|u'|)}^{-1}
[\cH (u',z'',t,x)-\langle z'', 
 D _{ z'' } \cH (u',z'',t,x)\rangle] \in J,
\end{equation}
\begin{equation}
                                       \label{9.8.4}
|\cH (u',z'',t,x)-\cH (u',z'',s,y)|
\leq N(|t-s|+|x-y|)(1+|z''|+|u'|),
\end{equation}
where $N$ is a constant independent of $u',z'',t,x,s,y$.

{\color{black}

We also need the following result in which assumption
\eqref{9.22.1} is crucial.

\begin{lemma}
                                         \label{lemma 5.3.1}
The function $\cH$ is locally Lipschitz continuous
with respect to $u'$ and at all points of its differentiability
with respect to $u'$ we have
$$
|\cH_{u'}(u',z'',t,x)|\leq N(1+|u'|+|z''|),
$$
where the constant $N$ is independent of $(u',z'',t,x)$.
\end{lemma}

\begin{proof} The reader might find many similarities 
of the argument below with 
the proof of Theorem 4.6 of \cite{Kr11}. It suffices to show
that there exist constants $N,\varepsilon_{0}>0$ such that
for all $u',v',(t,x)\in\bR^{d+1}$, $z''\in\bR^{m}$,
and $y''\in\bS$, with $|v'|\leq\varepsilon_{0}$, we have
$$
\max_{(f,l'')\in
B(u'+v',y'',t,x)}\big[(\bar{H}+K
+|u'+v'|)f+\lambda_{k}(l'')z_{k}''\big]
$$
\begin{equation}
                                                              \label{6.3.3}
\leq \max_{(f,l'')\in
B(u',y'',t,x)}\big[(\bar{H}+K+|u'|)f+\lambda_{k}(l'')z_{k}''\big]
+N|v'|(1+|u'|+|z''|).
\end{equation}
For simplicity of notation we drop the arguments $t,x$ below.
Fix $u',v',y''$.
Inequality \eqref{9.8.3} shows that there is $(f_{0},l''_{0})
\in C''$ such that
$$
(\bar{H}+K+|u'|)f_{0}+l''_{0ij}y''_{ij}=H(u',y'').
$$
For $t\in[0,1]$ and $(f,l'')\in B''$ define
$$
f_{t}(f)=(1-t)f_{0}+tf,\quad l''_{t}(l'')=(1-t)l''_{0}+tl''
$$
and  observe that since $C''$ lies in the interior of $B''$,
for any $t\in[0,1]$ 
$$
(f_{t}(f)-K_{0}'(1-t),l''_{t}(l''))\in B''.
$$
Now if $(f,l'')\in B(u'+v',y'')$, then
$$
I:=(\bar{H}+K+|u' |)[f_{t}(f) -K_{0}'(1-t)]
+l''_{tij}(l'')y''_{ij}
$$
$$
=t[(\bar{H}+K+|u'+v'|)  f+ l''_{ij}y''_{ij}]
+t(|u'|-|u'+v'|)f
$$
$$
 +(1-t)H(u',y'') -K_{0}'(1-t)
(\bar{H}+K+|u' |).
$$
Here the first term on the right is by definition less than $tH(u'+v',y'')
\leq tH(u',y'')+N'|v'|$, the second one is less that
$2K_{0}'|v'| $. Hence
$$
I\leq  H(u' ,y'')+(N'+2K_{0}')|v'|-K_{0}'(1-t)
(\bar{H}+K)\leq  H(u' ,y''),
$$
provided that 
\begin{equation}
                                                            \label{6.3.5}
(N'+2K_{0}')|v'|\leq K_{0}'(1-t)(\bar{H}+K).
\end{equation}
In particular, for those $v'$ and $t$ we have
$(f_{t}(f) -K_{0}'(1-t),l''_{t}(l''))\in B(u',y'')$ so that
$$
J:=\max_{(f,l'')\in B(u'+v',y'')}[(\bar{H}+K+|u'|)[f_{t}(f) -
K_{0}'(1-t)]
+\lambda_{k}(l''_{t}(l''))z''_{k}]
$$
$$
\leq \max_{(f,l'')\in
B(u',y'' )}\big[(\bar{H}+K+|u'|)f+\lambda_{k}(l'')z_{k}''\big].
$$

Furthermore, $\lambda_{k}$ are Lipschitz continuous and
$|\lambda_{k}(l''_{t}(l''))-\lambda_{k}(l'')|\leq N(1-t)$,
where $N$ depends only on $\delta$, $d$, and the Lipschitz
constants of $\lambda_{k}$. Also $|f_{t}(f)-f|\leq 4K_{0}'(1-t)$.
It follows that
$$
J\geq-N(1-t)(1+|u'|+|z''|)+
\max_{(f,l'')\in B(u'+v',y'')}[(\bar{H}+K+|u'|)f  
+\lambda_{k}(  l'' )z''_{k}]
$$
$$
\geq-N(1-t)(1+|u'|+|z''|)-2K_{0}'|v'|
$$
$$
+
\max_{(f,l'')\in B(u'+v',y'')}[(\bar{H}+K+|u'+v'|)f  
+\lambda_{k}(  l'' )z''_{k}],
$$
where $N$ is independent of $u',v',z'',y''$ (and $(t,x)$).
We thus have obtained \eqref{6.3.3} with
$N(1-t)(1+|u'|+|z''|)+N|v'|$ in place of 
$N|v'|(1+|u'|+|z''|)$ provided that \eqref{6.3.5}
holds. After taking (here we use that $K>0$)
$$
\varepsilon_{0}=K_{0}'(\bar{H}+K)/(N'+2K_{0}'),\quad (>0),
\quad
1-t=|v'|/\varepsilon_{0}\quad(\in[0,1])
$$
we come to the original form of \eqref{6.3.3} and the lemma is proved.

\end{proof}

Having representation \eqref{9.8.1} and having in mind
finite-differences make it natural to use the following
''monotone'' approximations of $H[v]$ and $P[v]$
with finite difference operators.} 
For $h>0$ and vectors $l$ introduce
$$
T_{h,l}\phi(x)=\phi(x+hl),\quad \delta_{h,l}=h^{-1}(T_{h,l}-1),
\quad \Delta_{h,l}=h^{-2}(T_{h,l}-2+T_{h,-l}).
$$
Also set {\color{black}(recall that $\cP$ is introduced in
\eqref{6.1.1})}
$$
{\color{black}\cH_{K}=\max(\cH,\cP-K),}\quad
P_{h}[v](t,x)=\cP ( 
\Delta_{h}v(t,x)),
$$
where
 $$
\Delta_{h} v=( \Delta_{h,l_{1}}v
,..., \Delta_{h,l_{m}}v).
$$
Similarly we introduce
$$
H_{h}[v](t,x)
=\cH (v(t,x),\delta_{h}v(t,x),\Delta_{h}v(t,x)),
$$
where  
$$
\delta_{h} v=( \delta_{h,e_{1}}v ,..., \delta_{h,e_{d}}v ), 
$$
and
$H_{K,h}[v]=\max(H_h[v],P_h[v]-K)$.

Owing to \eqref{9.8.1} we have $\cH(u',0,t,x)
=H(u',0,t,x)$ which in light of \eqref{9.8.2}
and Assumption \ref{assumption 9.23.1} (ii)
yields the following.
\begin{lemma}
                                          \label{lemma 9.14.1}
For all values of arguments
$$
\cH \leq  \cP -  (\hat\delta /2) 
  \sum_{k=1}^{m}|z''_{k}| +K_{0}|u'|+\bar{H}. 
$$
\end{lemma}

Introduce $B$ as the smallest closed
 ball containing $\Lambda$ {\color{black}(recall
its definition \eqref{6.1.2})} and set
$$ \Omega^{h}=\{x\in \Omega:x+hB\subset \Omega\} =\{x:\rho(x)>
\lambda h\}, $$ where $\lambda$ is the radius of $B$.

For $h>0$ such that $\Omega^{h}\ne\emptyset$  
consider the equation
\begin{equation}
                                              \label{2.25.3}
\partial_t v+H_{K,h}[v]=0\quad \text{in}\quad 
[0,T]\times\Omega^{h}  
 \end{equation}
 with   boundary condition
\begin{equation}
                                              \label{2.25.4}
v=g\quad\text{on}\quad \Big(\{T\}\times
 \Omega^{h}\Big)\cup\Big([0,T]\times(\bar \Omega\setminus
\Omega^{h})\Big) . 
 \end{equation}

In view of Picard's method  of  successive 
iterations, for any $h>0$, there
exists a unique bounded solution $v=v_{ h}$ of
\eqref{2.25.3}--\eqref{2.25.4}.
Furthermore, $\partial_{t}v_{h}(t,x)$ is bounded and
is continuous
with respect to $t$ for any $x$.  {\color{black}
A solution of \eqref{9.23.2}, whose existence
is claimed in Theorem \ref{theorem 10.5.1}, will be obtained
as the limit of a subsequence of $v_{h}$ as $h\downarrow0$.
Therefore, we need to have appropriate bounds on $\partial_{t}
v_{h}$ and the first- and second-order differences in $x$
of $v_{h}$.}

Below in this section 
by $h_{0}$ and $N$ with occasional indices we denote various
(finite positive) constants depending only on $\Omega$,
$\{l_{1},...,l_{m}\}$,
$d$, $K_{0}$, $T$, and $\delta$, 
unless specifically stated otherwise.

 Denote
 $$
 \Lambda_{1}= \Lambda,\quad
\Lambda_{n+1}=
 \Lambda_{n}+ \Lambda ,\quad n\geq1,
\quad \Lambda_{\infty}=\bigcup_{n}\Lambda_{n},
\quad \Lambda^{h}_{\infty}=h\Lambda_{\infty}. 
$$
Observe that the set of points in $\Lambda^{h}_{\infty}$ lying in any 
bounded domain is finite since the $l_{i}$'s have 
integral coordinates.

We need a particular case of Theorem 4.3 of \cite{DK}.
Let
$Q^{o}$ be a nonempty  subset
of $(0,T)\times \Lambda_{\infty}^{h}$, which is open
in the relative topology of
$(0,T)\times\Lambda_{\infty}^{h}$. We
introduce $\hat{Q}^{o}$ as the
set of points
$(t_{0},x_{0})\in(0,T]\times\Lambda_{\infty}^{h}$
for each of which there exists a sequence
$t_{n}\uparrow t_{0}$   such that
  $(t_{n},x_{0})\in Q^{o} $.
Observe that $Q^{o}\subset\hat{Q}^{o}$.
Also define
\begin{equation}
                                                   \label{6.10.1}
Q=\hat{Q}^{o}\cup\{(t,x+h\Lambda):(t,x)\in Q^{o}\}.
\end{equation}

For $x\in  \Lambda_{\infty}^{h}$
we denote by $Q^{o}_{|x}$ the $x$-section of $Q^{o}$:
$\{t:(t,x)\in Q^{o}\}$. Assume that
$$
Q^{o}\subset G:=\{(t,x)\in\Omega^{h}_{T}:(\hat\delta/2)\sum_{k=1}^{m}
|\Delta_{h,l_{k}}v_{h}(t,x)|> 
$$
\begin{equation}
                                             \label{9.25.1}
>\bar{ H}+K+K_{0}
\big(|v_{h}(t,x)|+M_{h}(t,x)  
\big)\},
\end{equation}
where
$$
M_{h}(t,x)=\sum_{k=1}^{m}|\delta_{h,l_{k}}
v_{h}(t,x)|,
$$
so that, owing to Lemma \ref{lemma 9.14.1},
$\partial_{t}v_{h}+P_{h}[v_{h}]-K\leq0$ in 
$[0,T]\times\Omega^{h}$ and
\begin{equation}
                                              \label{9.16.7}
\partial_t v_h+P_{h} [v_{ h}]=K\quad\text{in}\quad Q^{o}.
\end{equation} 

Also observe that, owing to the continuity of
$v_{h}$ in $t$, $G\cap [(0,T)\times \Lambda_{\infty}^{h}]$ is open
in the relative topology of
$\bR \times\Lambda_{\infty}^{h}$.

  To proceed with estimating $\partial_{t}v_{h}$
and second-order differences
of $v_{h}$ we introduce the following.
Take a function $\eta\in C^{\infty}(\bR^{d})$ with bounded
derivatives, such that $|\eta|\leq1$ and set
$\zeta=\eta^{2}$, 
$$
 |\eta'(x)|_{h} =\sup_{
k}|\delta_{h,l_ k}\eta  (x)|,\quad  |\eta''(x)|_{h}
=\sup_{ k}|
\Delta_{h,l_ k}\eta  (x)|, 
$$
 $$
 \|\eta'\|_{h}=\sup_{ \Lambda_{\infty}^{ h }}|\eta'
|_{h},\quad
 \|\eta''\|_{h}=\sup_{
\Lambda_{\infty}^{ h}}|\eta''
|_{h}.
$$

Here is a particular case of Theorem 4.3 of \cite{DK}
we need.

\begin{lemma}
                                         \label{lemma 9.13.2}
Assume that $ Q \subset [0,T]\times\Omega^{h}$. Then
there exists a constant $N=N(m,\delta)\geq1$ such that
on $Q^{o}$ for any $k=1,...,m$
$$
\zeta^{2}[ ( \Delta_{h,l_k}v_{h})^{-}]^{2}\leq \sup_{
 Q\setminus Q^{o} }\zeta^{2}[
( \Delta_{h,l_k}v_{h})^{-}]^{2} + N(\|\eta''\|_{h}+\|\eta'\|_{h}^{2})
\bar{W}_{k},
$$
where
$$
\bar{W}_{k}=\sup_{ Q }(|\delta_{h,l_k}v_{h}|^{2}+ 
|\delta_{h,-l_k}v_{h}|^{2}).
$$
\end{lemma}

{\color{black}
To investigate   $v_{h}$ near the boundary
we  need  part of  Lemma 8.8 of \cite{Kr11}.  
\begin{lemma}
                                           \label{lemma 4.1.1}
For any constants $\delta_{0},N_{0}\in(0,\infty)$
there exists a  constant  $N$,  
depending only on $ \delta_{0},N_{0},\Omega$,  and
there exists a function $\Psi\in C^{2}(\bar{Q})$
 such that   $N\rho\geq\Psi\geq\rho$ in $\Omega$ and
for all sufficiently small $h$ on $\Omega^{h}$
$$
\sum_{j=1}^{m}a_{j}\Delta_{h,l_{j}}\Psi + N_{0}
\sum_{j=1}^{m}|\delta_{h,l_{j}}\Psi |\leq-1,
$$
whenever $  \delta_{0} ^{-1}\geq a_{j}\geq \delta_{0} $.

\end{lemma}

\begin{remark}
Actually, the inequality  $N\rho\geq\Psi$
and the exact dependence of $N $ on the data  are  
not claimed in the statement of Lemma 8.8 of \cite{Kr11}.
These assertions follow directly from the proof. It may be also worth
noting that which $h$ are sufficiently small 
depend on the modulus of continuity of
the second-order derivatives of $\Psi$ which are
defined by the continuity properties of the second-order
derivatives of functions defining $\partial\Omega$.

\end{remark}

}
  
 \begin{lemma}
                                         \label{lemma 9.13.1}
There   are
constants  $h_{0}>0$  and $N$
 such that for all $h\in(0,h_{0}]$
\begin{equation}
                                              \label{9.13.1}
|v_{ h}-g|\leq N(\bar{ H}+K+
\|g\|_{C^{1,2}(\bar\Omega_T)}  )\rho,
\end{equation}
\begin{equation}
                                              \label{9.13.2}
  |\partial_t v_{ h}|
\leq N(\bar{M}_{h}+\bar{ H}+K
+\|g\|_{C^{1,2}(\bar\Omega_T)})
\end{equation}
on $\bar\Omega_T$, where
$
\bar{M}_{h}:=\sup_{[0,T]\times\Omega^{h}}M_{h}
$.
\end{lemma}

\begin{proof} To prove \eqref{9.13.1} observe that by 
Hadamard's formula {\color{black}(cf. \eqref{5.30.4})}
$$
0=\partial_{t}v_{h}+H_{K,h}[v_{h}]  =
\partial_{t}v_{h}
$$
$$
+\max\big(\cH(v_{h},\delta_{h}v_{h},
\Delta_{h}v_{h},t,x),P_{h}[v_{h}]-K\big)
-\max\big(\cH (v_{h},\delta_{h}v_{h},
0,t,x), -K\big)
$$
$$
+\max\big(\cH (v_{h},\delta_{h}v_{h},
0,t,x), -K\big) 
$$
\begin{equation}
                                                    \label{9.11.1}
=\partial_{t}v_{h}+\sum_{k=1}^{m}a_{k}\Delta_{h,l_{k}}v_{h}
+f(v_{h},\delta_{h}v_{h},t,x),
\end{equation}
where $a_{k}$ are some functions satisfying $ \hat{\delta} /2
\leq a_{k}\leq 2\hat{\delta}^{-1}$
and, owing to \eqref{9.8.3}, 
$f(v_{h},\delta_{h}v_{h},t,x)$ satisfies
\begin{equation}
                                                    \label{9.11.2}
|f|\leq 
N_{1} (\bar{H} +K +
|v_{h}|+{\color{black}\sum_{k=1}^{d}|\delta_{h,e_{k}}v_{h}|}) ,
\end{equation}
where $N_{1}=N( d)K'_{0} $.
This properly of $f$ implies that there exist functions
$b_{k}$, $k=1,...,d$, $c$, and $\theta$
with values in $[-N_{1},N_{1}]$
 such that
$$
f(v_{h},\delta_{h}v_{h},t,x)=c v_{h}+\sum_{k=1}^{d}
b_{k}\delta_{h,e_{k}}v_{h}+\theta(\bar{H}+K),
$$
so that $w_{h}(t,x):=v_{h}(t,x)\exp(N_{1}t)$ satisfies {\color{black}
$$
\partial_{t}w_{h}+L_{h}w_{h} +\theta
(\bar{H}+K)e^{N_{1}t}
=0
$$ 
in $[0,T]\times\Omega^{h}$,
where
$$
L_{h}w:=\sum_{k=1}^{m}a_{k}\Delta_{h,l_{k}}w 
+\sum_{k=1}^{d}
b_{k}\delta_{h,l_{k}}w +(c-N_{1}) w .
$$
}
After that
\eqref{9.13.1} for $h$ small enough {\color{black}  
follows in a standard way from the maximum principle
and the properties of $\Psi$ from Lemma \ref{lemma 4.1.1}.
To be more specific observe that
 for a constant
$N_{2}$   we have on $[0,T]\times (
 \Omega^{h}\cap\Lambda^{h}_{\infty})$ that
$$
\partial_{t}(ge^{N_{1}t})+L_{h}(ge^{N_{1}t}) 
\leq  N_{2}\|g\|_{c^{1,2}(\bar{\Omega}_{T})}=: N_{3}.
$$
Furthermore,
 $c-N_{1}\leq 0$
and for an appropriate choice of $\delta_{0},N_{0}$
and $N_{4}=N_{3}+N_{1}(\bar{H}+K)\exp(N_{1}T)$
$$
\partial_{t}(N_{4}\Psi)+L_{h}(N_{4}\Psi)+N_{3}+\theta
(\bar{H}+K)e^{N_{1}t}\leq0
$$ 
in $[0,T]\times (
 \Omega^{h}\cap\Lambda^{h}_{\infty})$. Hence, the function
$$
u_{h}=(v_{h}-g)e^{N_{1}t}-N_{4}\Psi
$$
satisfies 
$$
\partial_{t}u_{h}+L_{h}u_{h}\geq0
$$
in $[0,T]\times (
 \Omega^{h}\cap\Lambda^{h}_{\infty})$. Since the set
$\Omega^{h}\cap\Lambda^{h}_{\infty}$ has only finite  number
of points it follows from the maximum principle that
$$
u_{h}\leq 
\max\{u^{+}_{h}(T,x):x\in \Omega\cap\Lambda^{h}_{\infty} \}
$$
$$
+\max\{u^{+}_{h}(t,x):t\in[0,T],x\not\in
\Omega^{h}\cap\Lambda^{h}_{\infty},\exists \,k:
x-hl_{k}\in \Omega^{h}\cap\Lambda^{h}_{\infty}\}
$$
in $[0,T]\times (
 \Omega^{h}\cap\Lambda^{h}_{\infty})$. The conditions
imposed on $x$ inside the second maximum sign imply that
$x\in\Lambda^{h}_{\infty}$, $x\not\in
\Omega^{h}$, and $x\in \Omega\setminus
\Omega^{h}$. This along with the boundary condition
\eqref{2.25.4} leads us to the conclusion that
$$
u_{h}=(v_{h}-g)e^{N_{1}t}-N_{4}\Psi\leq0
$$
in  $[0,T]\times (
 \Omega^{h}\cap\Lambda^{h}_{\infty})$ and, owing to 
an obvious possibility of translations,
in  $[0,T]\times  
 \Omega^{h} $. By using \eqref{2.25.4} one more time we see
that, actually,
$$
v_{h}-g\leq N_{4}\Psi
$$
in $\bar{\Omega}_{T}$. This yields the needed estimate of
$v_{h}-g$ from above. Similarly one obtains it from below as well.

Passing to \eqref{9.13.2} and}
having in mind translations and the continuity
of $\partial_{t}v_{h}$ with respect to $t$
 we see that it suffices to prove
 \eqref{9.13.2} on $(0,T)\times (
\bar\Omega\cap\Lambda^{h}_{\infty})$. 
{\color{black}Recall that $G$ is  defined in \eqref{9.25.1}
and} introduce  
$$
Q^{o} =
\{(0,T)\times[\Omega^{h}
\cap\Lambda^{h}_{\infty}]\}\cap G.  
$$

Since $v_h$ satisfies \eqref{2.25.4},
estimate \eqref{9.13.2} obviously holds on 
$$
(0,T)\times(\bar\Omega\setminus 
\Omega^{h}) .
$$
On $(0,T)\times[\Omega^{h}
\cap\Lambda^{h}_{\infty}]\setminus   Q^{o}$, we have
\begin{equation}
                                                     \label{9.22.7}
(\hat\delta/2)\sum_k |\Delta_{h,l_k}v_{h}| 
\leq \bar{H}+K+K_{0}
\big(|v_{h} |+M_{h}
\big),
\end{equation}
which together with \eqref{9.13.1}, \eqref{9.11.1},  and
\eqref{9.11.2} implies that \eqref{9.13.2}  
 holds on
$(0,T)\times[\bar\Omega
\cap\Lambda^{h}_{\infty}]\setminus   Q^{o} $. Therefore, it
remains to establish  \eqref{9.13.2}
on $Q^{o}$ assuming that $Q^{o}\ne
\emptyset $.

Recall that \eqref{9.16.7} holds.
Furthermore,
every $x$-section of $Q^{o}$
is the union of open intervals  on which $\partial_{t}v_{h}$
is Lipschitz continuous   by virtue of \eqref{9.16.7}.
By subtracting the left-hand sides of
\eqref{9.16.7} evaluated at points $t$ and $t+\varepsilon$,
then transforming the difference by
using Hadamard's formula {\color{black}(as in \eqref{5.30.4})}, and finally dividing by
$\varepsilon$ and letting $\varepsilon\to0$,
 we get that there exist functions $a_{k} $ such that
$\hat{\delta}/2\leq a_{k}\leq2\hat{\delta}^{-1}$ and
on every $x$-section of $Q^{o}$ (a.e.)  we have
$$
\partial_t (\partial_t v_h) + a_k\Delta_{h,l_k} (\partial_t v_h) =0.
$$
By Lemma 4.2 of \cite{DK} this yields 
$$
\sup_{Q^{o}}|\partial_t
v_h|\leq\sup_{ (0,T] \times[\Omega\cap\Lambda^{h}_{\infty}
] \setminus  Q^{o}}
|\partial_t v_h|,
$$
which implies \eqref{9.13.2} on $Q^{o}$. The lemma 
is proved.
\end{proof}

{\color{black}
\begin{remark}
                                             \label{remark 6.9.1}
The fact that the first-order differences enter
the right-hand side of \eqref{9.13.2} reflects
a big difference between settings in this paper and in
\cite{DK} and \cite{Kr11} where it was possible to estimate
the first-order differences on the account of having
them in $P$ and then requiring from the start Lipschitz continuity
of $H$ with respect to $u'$. In our situation
the first-order differences will also enter estimates
of the second order differences and then will be excluded
from the right-hand sides by using interpolation,
which is somewhat more delicate than usual because
we could not obtain global  estimates 
of the second-order differences
  and only get estimates 
blowing up near the boundary.

\end{remark}

}

\begin{lemma}
                                          \label{lemma 9.16.1}
There  are
constants $h_{0}>0$ 
and $N$  such that for
all $h\in(0,h_{0}]$ and $r=1,...,m$ 
\begin{equation}
                                              \label{9.16.6}
 (\rho-6\lambda h) |\Delta_{h,l_{r}}  v_{ h}|\leq N(
\bar M_{h}+
\bar{H}+K
+\|g\|_{C^{1,2}(\bar\Omega_T)}) 
\end{equation} on
$[0,T]\times \Omega^{h}$ (remember that $\lambda$ is the radius of
$B$).
\end{lemma}

\begin{proof}
As in the proof of Lemma \ref{lemma 9.13.1} we will focus
on proving \eqref{9.16.6} in
$(0,T)\times[\Omega^{h}\cap
\Lambda_{\infty}^{h} ]$.  Then
  fix $r$ and define
$$
Q^{o}:=\{ 
(0,T)\times[\Omega^{3h}\cap
\Lambda_{\infty}^{h}]\}\cap G.
$$
{\color{black}For $Q$ from \eqref{6.10.1},} obviously,
$Q\subset[0,T]\times\Omega^{h}$. Next,
if $t\in (0,T)$, and $x\in \Omega^{h}\cap\Lambda^{h}_{\infty}$ is such
that $(t,x) \not\in Q^{o}$, then either $x\not\in \Omega^{3h}$, so that
$\rho(x)\leq 3\lambda h$ and \eqref{9.16.6} holds, or else $x \in
\Omega^{3h}$ but \eqref{9.22.7} is valid,
 in which case \eqref{9.16.6} holds again.

Thus we need only prove \eqref{9.16.6} on $Q^{o}$ assuming, of course,
that $Q^{o} \ne\emptyset$. We know that 
  \eqref{9.16.7}   holds and the left-hand side of
\eqref{9.16.7} is nonpositive  in $Q\setminus Q^{o}$.

To proceed further observe
 a standard fact that there are constants $\mu_{0}\in(0,
1]$
and $N\in[0,\infty)$ depending only on $\Omega$  such that
 for any $\mu\in(0,\mu_{0}]$ there exists an $\eta_{\mu}
\in C^{\infty}_{0}(\Omega)$ satisfying
$$
\eta_{\mu}=1\quad\text{on}\quad \Omega^{2\mu},\quad
\eta_{\mu}=0\quad\text{outside}\quad \Omega^{\mu},
$$
\begin{equation}
                                                    \label{9.22.6}
|\eta_{\mu}|\leq1,\quad |D\eta_{\mu}|\leq
N/\mu,\quad|D^{2}\eta_{\mu}|\leq N/\mu^{2}.
\end{equation}
By  Lemma \ref{lemma 9.13.2} on $Q^{o}\cap\Omega^{2\mu}_{T}$
$$
 [ ( \Delta_{h,l_r}v_{h})^{-}]^{2}\leq \sup_{
 Q\setminus Q^{o} }\eta_{\mu}[
( \Delta_{h,l_r}v_{h})^{-}]^{2} + N
\mu^{-2}\bar M_{h}^{2}.
$$
While estimating the last supremum
we will only concentrate on $h_{0}\leq\mu_{0}/3$
and $\mu\in[ 3h,\mu_{0}]$, when
$\eta_{\mu}=0$ outside $\Omega^{3h}$. In that case, for any $(s,y)\in
  Q\setminus Q^{o}$, either $y\notin\Omega^{3h}$ implying that
$$
\eta_{\mu} [ ( \Delta_{h,l_{r}}v_{h})^{-}]^{2}(s,y)=0,
$$
or  $y \in\Omega^{3h}\cap \Lambda_{\infty}^{h}$ but 
\eqref{9.22.7} holds at $(s,y)$,
or else ($y \in\Omega^{3h}\cap \Lambda_{\infty}^{h}$ and 
$(s,y)\notin Q^{o}$ and)
there is a sequence $s_{n}\uparrow s$ such that
$(s_{n},y)\in Q^{o}$. 

The third possibility splits into two cases:
1) $s=T$, 2) $s<T$. In case 1 we have
$$
|\Delta_{h,l_{r}}
v_{ h}(s,y)|=|\Delta_{h,l_{r}}
g(s,y)|\leq N\|g\|_{C^{1,2}(\bar\Omega_T)}.
$$
In case 2, estimate \eqref{9.22.7}
holds by  the definition of $Q^{o}$.

It follows that as long as $h\in(0,h_{0}]$, $(t,x)\in Q^{o}
\cap\Omega^{2\mu}_{T}$, and  $\mu\in
[3h,\mu_{0}]$ we have   
\begin{equation}
                                                   \label{9.22.07}
 ( \Delta_{h,l_{r}}v_{h})^{-}(t,x)\leq
 N
\mu^{-1}(\bar{H}+K
+\|g\|_{C^{1,2}(\bar\Omega_T)}+\bar M_{h}).
\end{equation}

If $(t,x)\in Q^{o}$ and $x$ is such that $\rho(x)\geq 6\lambda h$, take
$\mu=\mu_{0}\wedge(\rho(x)/(2\lambda))$, which is bigger than $3h$
for $h\in(0,h_{0}]$ since $h_{0}\leq\mu_{0}/3$.
 In that case also $ \rho(x)\geq 2\lambda \mu$, so that 
$x\in \Omega^{2\mu}$ and
we conclude from \eqref{9.22.07} that
$$ 
( \Delta_{h,l_{r}}v_{h})^{-}(t,x)
\leq N\mu^{-1}(\bar{H}+K
+\|g\|_{C^{1,2}(\bar\Omega_T)}+\bar M_{h}).
$$
Furthermore, still in case
$\mu=\mu_{0}\wedge(\rho(x)/(2\lambda))$, as is easy to see,
there is a constant $N$, depending only on $\lambda,\mu_{0}$,
and the diameter of $\Omega$, such that $\mu^{-1}\leq
N\rho^{-1}(x)$. Therefore,
$$ \rho(x)( \Delta_{h,l_{r}}v_{h})^{-}(t,x)
\leq N(\bar{H}+K
+\|g\|_{C^{1,2}(\bar\Omega_T)}+\bar M_{h}),
$$
$$
(\rho(x)-6\lambda h)( \Delta_{h,l_{r}}v_{h})^{-}(t,x) \leq N
(\bar{H}+K
+\|g\|_{C^{1,2}(\bar\Omega_T)}+\bar M_{h})
$$
for $(t,x)\in Q^{o}$ such that
$\rho(x)\geq 6\lambda h$. However, the second relation here
is obvious for $\rho(x)\leq 6\lambda h$.

As a result of all the above arguments we see that 
\begin{equation}
                                                   \label{9.22.8}
(\rho-6\lambda h)( \Delta_{h,l_{r}}v_{h})^{-}\leq N(
\bar{H}+K
+\|g\|_{C^{1,2}(\bar\Omega_T)}+\bar M_{h}) 
\end{equation}
 holds
in $(0,T)\times[\Omega^{h}\cap  \Lambda_{\infty}^{h}]$
  for any $r$ whenever $h\in(0,h_{0}]$.  

Finally, since $\partial_t v_h+P_{h}[v_{h}]\leq K$ in $(0,T)\times \Omega^{h}$, we have that
$$
2\hat{\delta}^{-1}\sum_{r}(\Delta_{r}v_{h})^{ +} 
\leq-\partial_t v_h+( \hat\delta/2)
\sum_{r}(\Delta_{r}v_{h})^{ -}+K,
$$
which after being multiplied by $\rho-6h$ along
with \eqref{9.22.8}  and \eqref{9.13.2}
 leads to \eqref{9.16.6}
on $(0,T)\times[\Omega^{h}\cap  \Lambda_{\infty}^{h}]$. Thus, as
  is explained at the beginning of the proof,
  the lemma is proved.
\end{proof}

Our final estimates hinge on the first-order difference
estimates.
\begin{lemma}
                                          \label{lemma 9.22.3}
 
There is a constant $N$  such that for all sufficiently
small $h>0$ the estimates
\begin{equation}
                                              \label{9.18.9}
|v_{h}|,|\partial_{t}v_{h}|,|\delta_{h,l_{k}}v_{h}|,
( \rho -6\lambda h)|\Delta_{h,l_{k}}v_{h}|\leq N
(\bar{H}+K
+\|g\|_{C^{1,2}(\bar\Omega_T)})
\end{equation}
hold  in $[0,T]\times \Omega^{h}$ for all $k$.
\end{lemma}
 
\begin{proof} The first estimate in \eqref{9.18.9} is obtained
in Lemma \ref{lemma 9.13.1}. Owing to Lemmas \ref{lemma 9.13.1}
and \ref{lemma 9.16.1},   the remaining estimates
would follow if we can prove that
\begin{equation}
                                                   \label{9.22.08}
 |\delta_{h,l_{k}}v_{h}|\leq N
(\bar{H}+K
+\|g\|_{C^{1,2}(\bar\Omega_T)})
\end{equation}
  in $[0,T]\times \Omega^{h}$ for all $k$.

We are going to use interpolation inequalities.
Note that if we have a function $u(i)$ on a set $-r+1,...,
0, 1,...,r$,
where $r\geq 2$ is an integer, which satisfies
\begin{equation}
                                                  \label{3.17.2}
u(i+1)-2u(i)+u(i-1)\geq -N_{1}
\end{equation}
for $i=-r+2,...,r-1$, where $N_{1}$ is a constant,
 then 
$$
u(i+1)-u(i)\geq u(i)-u(i-1)-N_{1}.
$$
It follows that $w(i):=u(i+1)-u(i)+N_{1}i$ is an increasing
function of $i=-r+1,...,r-1$.
In particular,
$$
u(1)-u(0)=w(0)\leq\frac{1}{r-1}\sum_{i=1}^{r-1}w(i)
$$
$$
=\frac{1}{r-1}\sum_{i=1}^{r-1}(u(i+1)-u(i)+N_{1}i)
=\frac{1}{r-1}(u(r)-u(1))+\frac{1}{2}N_{1}r.
$$ 
On the other hand,
$$
u(1)-u(0)\geq\frac{1}{r-1}\sum_{i=-r+1}^{-1}(u(i+1)-u(i)+N_{1}i)
$$
$$
=\frac{1}{r-1}(u(0)-u(-r+1))-\frac{1}{2}N_{1}r.
$$
 
It follows that 
$$
|u(1)-u(0)|\leq \frac{1}{2}N_{1}r+\frac{2}{r-1}\max\{|u(i)|:i=-r+1,...,r\},
$$
and for any function $w$ (use that $(r-1)^{-1}\leq2r^{-1}$
for $r\geq2$)
\begin{equation}
                                             \label{9.24.7}
|w(1)-w(0)|\leq  \frac{r}{2} \max_{|i|\leq r}
|w(i+1)-2w(i)+w(i-1)|+\frac{4}{r}\max_{|i|\leq r} |w(i)|.
\end{equation}
 
Now fix an $\varepsilon\in(0,1]$  
and set
$$
n(\varepsilon)=10/\varepsilon.
$$
Observe that if $x\in\Omega^{n(\varepsilon)h}$ 
and we take $r=[(\varepsilon\rho(x)-6\lambda h) (2\lambda h)^{-1}]$
($[a]$ is the integer part of $a$), then
$r\geq2$ and 
\begin{equation}
                                             \label{11.7.3}
\varepsilon[ \rho(x+ihl_{k})-6\lambda h] \geq r\lambda h
\end{equation}
  for $|i|\leq r$ since
 $
\rho(x+ihl_{k})\geq\rho(x)-\lambda rh
 $
and
$$
\varepsilon\rho(x)-(1+\varepsilon)r\lambda h
\geq\varepsilon\rho(x)-2r\lambda h
\geq 6\lambda h.
$$
In particular, $x+ihl_{k}\in\Omega^{h}$ for $|i|\leq r$
and it makes sense applying \eqref{9.24.7}
to $w(i)=v_{h}(t,x+ihl_{k})-g(t,x+ihl_{k})$ with $t\in(0,T)$,
which yields
$$
|\delta_{h,l_{k}}(v_{h}-g)(t,x)|
\leq \frac{1}{2}rh\max_{|i|\leq r}|
\Delta_{h,l_{k}}(v_{h}-g)(t,x+ihl_{k})|
$$
\begin{equation}
                                             \label{9.13.6}
+\frac{4}{rh}
\max_{|i|\leq r}|
 (v_{h}-g)(t,x+ihl_{k})|.
\end{equation}

Also notice that for $x\in \Omega^{n(\varepsilon)h}$
$$ 
2r\lambda h\geq \varepsilon\rho(x)-8\lambda h,\quad
10\lambda h<\varepsilon\rho(x),\quad 
10r\lambda h\geq \varepsilon\rho(x),
$$

\begin{equation}
                                          \label{11.8.1}
\rho(x+ihl_{k})\leq \rho(x)+r\lambda h\leq
rh(10\lambda\varepsilon^{-1}+\lambda)\leq rh11\lambda
\varepsilon^{-1}.
\end{equation}
Estimates \eqref{11.7.3} and \eqref{11.8.1}
allow us to derive from \eqref{9.13.6} that
$$
|\delta_{h,l_{k}}(v_{h}-g)(t,x)|
\leq N\varepsilon\max_{|i|\leq r}[ \rho(x+ihl_{k})-6\lambda h]|
\Delta_{h,l_{k}}(v_{h}-g)(t,x+ihl_{k})|
$$
$$
+N\varepsilon^{-1}
\max_{|i|\leq r}\rho(x+ihl_{k})^{-1}|
 (v_{h}-g)(t,x+ihl_{k})|,
$$
which along with
Lemmas \ref{lemma 9.13.1} and \ref{lemma 9.16.1}
shows that for all sufficiently small $h$,
$\varepsilon\in(0,1]$, and $x\in\Omega^{n(\varepsilon)h}$
$$
|\delta_{h,l_{k}}(v_{h}-g)(t,x)|\leq
N\varepsilon(\bar{M}_{h}+
\bar{H}+K
+\|g\|_{C^{1,2}(\bar\Omega_T)}) 
$$
$$
+ N\varepsilon^{-1}(\bar{H}+K+ 
\|g\|_{C^{1,2}(\bar\Omega_T)}  ).
$$  

 Hence, for all sufficiently
small $h$ we have 
$$
\bar{M}_{h}=\sup_{[0,T]\times \Omega^{h}}\sum_{k=1}^{m}
|\delta_{h,l_{k}}v_{h}|\leq
N_{1}\varepsilon(
\bar{M}_{h}+
\bar{H}+K
+\|g\|_{C^{1,2}(\bar\Omega_T)}) 
$$

 $$
+ N\varepsilon^{-1}(\bar{H}+K+
\|g\|_{C^{1,2}(\bar\Omega_T)}  )
+\sup_{[0,T]\times(\Omega^{h}\setminus
 \Omega^{n(\varepsilon)h})}\sum_{k=1}^{m}
|\delta_{h,l_{k}}(v_{h}-g)|,
$$
where the last term is dominated by 
$$
Nn(\varepsilon)(\bar{H}+K
+\|g\|_{C^{1,2}(\bar\Omega_T)})
$$
in light of  \eqref{9.13.1}.
To finish proving \eqref{9.22.08} it now remains only
  pick and fix $\varepsilon\in(0,1]$ so that $N_{1}\varepsilon\leq1/2$.
The lemma is proved.
\end{proof}

\mysection{Proof of Theorem \protect\ref{theorem 10.5.1}}
                                             \label{section 9.21.2}

{\color{black} In contrast with the proofs 
in \cite{DK} and
\cite{Kr12.2} of the statements similar to 
Theorem \ref{theorem 10.5.1}, here the proof consists
of two parts. The first part goes indeed very much like
in \cite{DK} and \cite{Kr12.2}
but only in case that $H$ is independent of $u'_{0}$.
 This happens because while getting uniform in $h$
estimates of the modulus of continuity of $v_{h}$,
we apply a finite-difference operator $T_{h,l}-1$
  to the equation
 and obtain an equation
for   $(T_{h,l}-1)v_{h}$ with coefficients
controlled by $v_{h}$, $\delta_{h}v_{h}$, and
$\Delta_{h}v_{h}$. This is  harmless if the coefficient
of $(T_{h,l}-1)v_{h}$ turns out to be bounded.
Observe that this coefficient
is basically the derivative of $\cH$ with respect to
$u'_{0}$ and it is indeed under control
in the situation of \cite{DK} and
\cite{Kr12.2} or when $\Omega=\bR^{d}$. Note that
in the estimate of this coefficient the second-order 
differences
of $v_{h}$ enter (see Lemma \ref{lemma 5.3.1})
and 
 in  the case
of bounded domain
the estimate blows up
near the boundary. That is why we first prove
Theorem \ref{theorem 10.5.1} when $H$ is independent
of $u'_{0}$, so that we can set $u'_{0}=0$ in $\cH$
and then we forget about $\cH$ and prove
Theorem \ref{theorem 10.5.1} in full generality
by using the Banach fixed point theorem.

Here is an estimate of the modulus
of continuity of $v_{h}$ useful in the particular case
that $H$ is independent of $u'_{0}$. In the following
lemma   \eqref{9.8.4}
 plays  a crucial role  and in \eqref{9.8.4} only the
Lipschitz continuity in $x$ is needed. By the way, notice that
as is easy to see
all the results in Section \ref{section 5.3.1} are valid for the solution
$v^{0}_{h}$ of the equation
$$
\partial_{t}v+\cH^{0}_{K}(\delta_{h}v ,\Delta_{h}v ,t,x) =0
$$
in $[0,T]\times\Omega^{h}$ with the same boundary condition
\eqref{2.25.4}, where
$$
\cH^{0}_{K}(\delta_{h}v ,\Delta_{h}v ,t,x)=
\max(\cH(0,\delta_{h}v ,\Delta_{h}v ,t,x),P_{h}[v ]-K)
$$

}
\begin{lemma}
                                        \label{lem11.24}

There are constants $h_{0}>0$ and $M$ {\color{black}
and there is a function $\omega_{1}(h)$, $h>0$,
such that $\omega_{1}(0+)=0$ and}
 for
all $h\in(0,h_{0}]$, $t \in [0,T]$, and $x,y \in\Omega$, we have
\begin{equation}
                                      \label{6.4.2}
|v^{0}_{h}(t,x)-v^{0}_{h}(t,y)|\leq
M(|x-y|+{\color{black}\omega_{1}(h)}).
\end{equation}
\end{lemma}
{\color{black}

\begin{proof} We closely follow the main idea of 
 the proof of Corollary 2.7 of \cite{Kr12.1} which  is
about elliptic equations.
Fix an $l\in\bR^{d}$ such that $|l|\leq1$ and define
$$
w_{h}(t,x)= v^{0}_{h}(t,x+hl)-v^{0}_{h}(t,x) .
$$
This function is well defined in $[0,T]\times\Omega^{h}$
(since $\lambda>1$).
Then observe that in $[0,T]\times\Omega^{h}$
$$
0=\partial_{t}w_{h}(t,x)+I_{h}(t,x)+J_{h}(t,x)+K_{h}(t,x),
$$
where
$$
 I_{h}(t,x)= \cH^{0}_{K}( \delta_{h}v^{0}_{h}(t,x+hl),
\Delta_{h} v^{0}_{h}(t,x+hl),t,x+hl)
$$
$$
-
\cH^{0}_{K}( \delta_{h}v^{0}_{h}(t,x+hl ),
\Delta_{h} v^{0}_{h}(t,x ),t,x+hl) ,
$$

$$
 J_{h}(t,x)= \cH^{0}_{K}( \delta_{h}v^{0}_{h}(t,x+hl ),
\Delta_{h} v^{0}_{h}(t,x ),t,x+hl)
$$
$$
-
\cH^{0}_{K}( \delta_{h}^{0}v_{h}(t,x  ),
\Delta_{h} v^{0}_{h}(t,x ),t,x+hl) ,
$$

$$
 K_{h}(t,x)= 
\cH^{0}_{K}( \delta_{h}v^{0}_{h}(t,x ),
\Delta_{h} v^{0}_{h}(t,x ),t,x+hl)
$$
$$
-\cH^{0}_{K}( \delta_{h}v^{0}_{h}(t,x ),
\Delta_{h} v^{0}_{h}(t,x ),t,x ) .
$$
As a few times in the past Hadamard's formula
allows us to conclude that there exist  functions
$a_{hk}(t,x)$, $k=1,...,m$, such that 
$\hat{\delta}/2\leq a_{hk}\leq2\hat{\delta}^{-1}$ and
in $[0,T]\times\Omega^{h}$
$$
I_{h}=a_{hk}\Delta_{h,l_{k}}w_{h}.
$$
According to Lemma \ref{lemma 5.3.1}
in $[0,T]\times\Omega^{h}$ we have 
$$
|J_{h}|
\leq N \sum_{k=1}^{d}|\delta_{h,e_{k}}w_{h}| 
\big[1 +\sum_{k=1}^{m}|
\Delta_{h,l_{k}} v^{0}_{h}(t,x )|
$$
$$
+\sum_{k=1}^{d}(|\delta_{h,e_{k}}v^{0}_{h}(t,x)|+
|\delta_{h,e_{k}}v^{0}_{h}(t,x+hl )|)\big],
$$ 
where and below by $N$ with occasional indices
we denote constants independent of $h$.
As far as $K_{h}$ is concerned, by \eqref{9.8.4}
$$
|K_{h}|\leq Nh \big[1 +
\sum_{k=1}^{d} |\delta_{h,e_{k}}v^{0}_{h}(t,x)|
+\sum_{k=1}^{m}|
\Delta_{h,l_{k}} v^{0}_{h}(t,x )|\big].
$$

The above estimates of $J_{h}$ and $K_{h}$
 along with Lemma \ref{lemma 9.22.3} show  that
$$
J_{h}= b _{hk}\delta_{h,e_{k}}w_{h},
\quad K_{h}=f_{h}h 
$$
with appropriate functions $ b_{hk},f_{h}$ which satisfy
the inequality
$$
 \sum_{k=1}^{d}|b_{hk}|+|f_{h}|\leq N_{1}/\rho
$$
in $[0,T]\times\Omega^{10 h}$
for sufficiently small $h>0$. Thus,
\begin{equation}
                                                         \label{6.3.7}
\partial_{t}w_{h}+a_{hk}\Delta_{h,l_{k}}w_{h} +b_{hk}\delta_{h,e_{k}}w_{h}
+f_{h}h =0
\end{equation}
in $[0,T]\times\Omega^{10 h}$ for sufficiently small $h>0$. We take
$\varepsilon\geq10h$ and notice that, due to \eqref{9.13.1}
and the fact that $w_{h}(T,x)
=g(T,x+hl)-g(T,x)$,  on 
\begin{equation}
                                                         \label{6.3.8}
\big(\{T\}\times\Omega^{h}\big)\cup
\big([0,T]\times[\Omega^{h}\setminus\Omega^{\varepsilon}]\big)
\end{equation}
we have  
 $
|w_{h}|\leq N_{2} \varepsilon ,
 $
  where the constant $N_{2}$  is independent of $\varepsilon$
(and  $h$).
It follows that the function
$$
\bar{w}_{h}(t,x)=w_{h}(t,x)-N_{2} \varepsilon 
$$
is negative on \eqref{6.3.8} and on $[0,T]\times\Omega^{\varepsilon}$
satisfies \eqref{6.3.7}. 
On the other hand, the function $u= 
e^{N_{1}(T-t)/\varepsilon}h $ is nonnegative on \eqref{6.3.8}
and as is easy to check  on $[0,T]\times\Omega^{\varepsilon}$
satisfies
$$
\partial_{t}u+a_{hk}\Delta_{h,l_{k}}u+b_{hk}\delta_{h,e_{k}}u
+f_{h}h \leq0.
$$
By the maximum principle in $[0,T]\times\Omega^{h}$, if
 $\varepsilon\geq 10h$, then
$$
w_{h}\leq N_{2}\varepsilon +
e^{N_{1}(T-t)/\varepsilon}h . 
$$
In other words if $x,y\in \Omega$, $|x-y|\leq h$, and one of $x$ or $y$
is in $\Omega^{h}$ and $t\in[0,T]$, then
\begin{equation}
                                      \label{6.4.1}
|v^{0}_{h}(t,x)-v^{0}_{h}(t,y)|\leq 
\min_{\varepsilon\geq 10h}[
N_{2}\varepsilon +
e^{N_{1} T /\varepsilon}h]=:\omega_{2}(h).
\end{equation}

Obviously, $\omega_{2}(0+)=0$ and if both $x,y\in\Omega\setminus\Omega^{h}$
and $|x-y|\leq h$, then \eqref{6.4.2} holds with $\omega_{1}(h)=
\omega_{2}(h)+N_{1}h$, where $N_{1}$ responsible for the boundary condition
is independent of $h$, $t$, $x$, and $y$.

In case $|x-y|\geq h$ and $h$ is sufficiently small,
 owing to the smoothness of $\Omega$, one
can find points $x^{1},...,x^{n}\in h\bZ^{d}\cap\Omega$,
such that $|x-x^{1}|,|x^{n}-y|\leq kh$, $x^{i+1}-x^{i}
\in\{\pm e_{1},...,\pm e_{d}\}$
for $i=1,...,n-1$,  $n\leq N|x-y|$, and
$k\in\{1,2,...\}$,
where $N$ and $k$ depend  only on $\Omega$. Then one derives 
\eqref{6.4.2} from the above result and from
estimate \eqref{9.18.9} which, in particular,
gives an estimate of $v^{0}_{h}(t,x^{i+1})-
v^{0}_{h}(t,x^{i})$. The lemma is proved.
\end{proof}
}

  {\bf Proof of Theorem \ref{theorem 10.5.1}}.
{\color{black} First we assume that $H$ is independent
of $u'_{0}$. Then}
in what concerns the first assertion of Theorem
\ref{theorem 9.12.1} and estimates \eqref{9.22.3}
one derives them in the same way 
 as Theorem 5.2 in \cite{DK} is proved
{\color{black} relying on the properties of $v^{0}_{h}$.  

In the general case we use the Banach fixed point
theorem. To start we take a Lipschitz continuous
with respect to $(t,x)$ function $w(t,x)$
defined in $ \bar{\Omega}_{T}$ and equal to $g$
on the parabolic boundary of this set, and introduce the function
$$
H^{w}(u,t,x)=H(w(t,x) ,u'_{1},...,u'_{d},u'',t,x).
$$
Obviously, $H^{w}$ satisfies 
Assumption \ref{assumption 9.23.1}  with
$\delta/2$ in place of $\delta$ and $\bar{H}^{w}\leq\bar{H}
+  K_{0}  \bar{w}$ in place of $\bar{H}$,
where
$$
\bar{w}=\sup_{\Omega_{T}}|w(t,x) |.
$$
The function $H^{w}$ also satisfies \eqref{9.10.2}
and \eqref{3.6.4}
if we replace $\bar{H}$ with $\bar{H}
+ (K_{0}+1) \bar{w}$.
Finally, $H^{w}$
satisfies
\eqref{9.22.1} (with the same $N'$)
and \eqref{9.31.2} (with a different one).

By the above the equation
\begin{equation}
                                                       \label{6.4.4}
\partial_{t}v+\max(H^{w}(Dv,D^{2}v,t,x),P[v]-K)=0
\end{equation}
in $\Omega_{T}$ with boundary condition $v=g$
on $\partial'\Omega_{T}$ has a solution
 $v^{w}\in
C(\bar{\Omega}_T)\cap W^{1,2}_{\infty,\text{loc}}(\Omega_{T})$.
In addition,  
\begin{equation}
                                             \label{6.5.1}
|v^{w}|, |D v^{w}|,\rho |D^{2} v^{w} |,|\partial_t v^{w} |
\leq N(\bar{H}+\bar{w} +K
+\|g\|_{ W^{1,2}_{\infty}(\Omega_{T})})  
\end{equation}
in $\Omega_T$ (a.e.), where   $N$ is a constant depending only on $\Omega$,
$T$, $K_{0}$, and $\delta$. Due to the Lipschitz continuity of $H^{w}$
and  parabolic Alexandrov maximum principle,
the solution is unique, so that the notation $v^{w}$ is valid.

Next, 
$$
H^{w}(Dv^{w},D^{2}v^{w},t,x)=H (0,Dv^{w},D^{2}v^{w},t,x)
+cw,
$$
where 
$$
c=\frac{1}{w }[H (w ,Dv^{w},D^{2}v^{w},t,x)
-H (0,Dv^{w},D^{2}v^{w},t,x)]
\quad(0^{-1}0:=0)
$$
and owing to \eqref{9.22.1} we have $|c|\leq N'$. As  has already been seen
before (cf. the proof of Lemma \ref{lemma 9.13.1})
 this allows us to write
$$
\partial_{t}v^{w}+a_{ij}D_{ij}v^{w}+b_{i}D_{i}v^{w}
+c'w +f=0,
$$
where $a$ is an $\bS_{\check\delta}$-valued function
($\check{\delta}$ is introduced in
Remark \ref{remark 6.5.1}),
 $|b|\leq K_{0}$, $|c'|\leq|c|\leq N'$, $|f|\leq\bar{H}+K$.
By the maximum principle
$$
|v^{w}(t,x)|\leq N'\int_{t}^{T}
\sup_{x\in\Omega}|w(x,s)|\,ds+ T(\bar{H}+K)
+\sup_{\Omega_{T}}|g |
$$
in $\Omega_{T}$. 
It follows that if 
$$
 |w(t,x)|\leq  (T(\bar{H}+K)+\sup_{\Omega_{T}}|g|)
e^{N'(T-t)}=:\hat{w}(t),
$$
then the same inequality holds for $v^{w}$.

We now introduce $S$ as the subset of $C(\bar{\Omega}_T)
\cap W^{1,2}_{\infty,\text{loc}}(\Omega_{T})$
of functions $w$ such that $|w|\leq\hat{w}$ and
$$
|w|, |D w|,\rho |D^{2}w|,|\partial_t w |
\leq N(\bar{H}+\hat{w}(0) +K
+\|g\|_{ W^{1,2}_{\infty}(\Omega_{T})})  
$$
in $\Omega_T$ (a.e.), where   $N$ is the constant from \eqref{6.5.1}.
Obviously $S$ is a closed set and the mapping $R:w\to Rw:=v^{w}$
maps $S$ into $S$. Furthermore, if $u,w\in S$, then 
$$
H(u,DRu,D^{2}Ru)-H(w,DRw,D^{2}w)
$$
$$
=
a_{ij}D_{ij}(Ru-Rw)+b_{i}D_{i}(Ru-Rw)
+c( u- w),
$$
where $a$ is an $\bS_{\delta/2}$-valued function, and due to
\eqref{9.22.1} also $|b|\leq N'$, $|c|\leq N'$ (we allow ourselves
the liberty to use the same letters $a,b,c$ for objects which may
be different). Hence
$$
\partial_{t}(Ru-Rw)+a _{ij}D_{ij}(Ru-Rw)+b _{i}D_{i}(Ru-Rw)
+c ( u- w)=0.
$$
  By the maximum principle it follows that
$$
|(Ru-Rw)(t,x)|\leq N'\int_{t}^{T}\sup_{x\in\Omega}|(u-w)(s,x)|\,ds 
$$
in $\Omega_{T}$, which implies that there exists an integer $n $ such that $R^{n}$
is a contraction of $S$. By the Banach fixed point theorem
there exists $v\in S$ such that $Rv=v$. 

In particular, this proves the first assertion of Theorem
\ref{theorem 9.12.1} in the general case and in light of 
\eqref{6.5.1} shows that to prove \eqref{9.22.3}
it only remains to prove that 
\begin{equation}
                                                                \label{6.5.4}
\sup_{\Omega_{T}}|v|\leq e^{K_{0}T}(T\bar{H} +\sup_{\Omega_{T}}|g|).
\end{equation}   

Take $F_{K}$ from Remark \ref{remark 6.5.1}}
and notice that since $|F_{K}(u',0,t,x)|\leq \bar{H} 
+K_{0}|u'|$, there exist functions $b_{1},...,b_{d}$,
 $c$, and $f$ such that
$$
|b_{i}|,|c|\leq K_{0},\quad |f|\leq\bar{H},
$$
$$
 F_{K}(v(t,x),Dv(t,x),0,t,x)=b_{i}(t,x)D_{i}v(t,x)
+c(t,x)v(t,x)+f(t,x),
$$
so that
$$
0=\partial_{t}v(t,x)+ F_{K}[v](t,x)-F_{K}(v(t,x),Dv(t,x),0,t,x)
$$
$$
+b_{i}(t,x)D_{i}v(t,x)
+c(t,x)v(t,x)+f(t,x),
$$
\begin{equation}
                                                                \label{6.5.5}
\partial_{t}v+a_{ij}D_{ij}v+b_{i}(t,x)D_{i}v(t,x)
+c(t,x)v(t,x)+f(t,x)=0,
\end{equation}
where $(a_{ij})$ is an $\bS_{\check{\delta}}$-valued function.
{\color{black} 
  By the maximum principle
$$
|v(t,x)|\leq K_{0}\int_{t}^{T}\sup_{x\in\Omega}|v(s,x)|\,ds
+T \bar{H} +\sup_{\Omega_{T}}|g|,
$$
and   Gronwall's inequality yields \eqref{6.5.4}.

}

To prove \eqref{1.13.2} observe that
$$
 \max(H(v(t,x),Dv(t,x),u'',t,x),P(u'')-K)=
  P(u'')+G(u'',t,x),
$$
where
$$
G(u'',t,x)=(H(v(t,x),Dv(t,x),u'',t,x)-P(u'')+K) ^+ -K.
$$ 
Furthermore, in light of \eqref{9.22.3} and \eqref{3.6.4}
$$
|G(u'',t,x)|\leq(H(v(t,x),Dv(t,x),u'',t,x)-P(u'')+K) ^+  +K
$$
\begin{equation}
                                              \label{9.24.6}
\leq \bar{H}+K_{0}\big(|v(t,x)|+|Dv(t,x)|\big)+2K\leq N,
\end{equation}
where $N$ is a constant like the right-hand side of \eqref{9.22.3}.
Then set
$$
G(t,x)=G(D^{2}v(t,x),t,x)
$$
and observe that our function $v$ satisfies the equation
\begin{equation}
                                              \label{9.24.5}
\partial_{t}u(t,x)+P(D^{2}u(t,x))+G(t,x)=0
\end{equation}

 Since $P$ is convex with respect to
$u''$ and $G( t,x)$ is bounded, due to Theorem 1.1 of
\cite{DKL} there is a unique solution $u\in
W^{1,2}_p(\Omega_T)$ of \eqref{9.24.5}
 with  boundary condition
$u=g$ on $\partial'\Omega_T$. By uniqueness of
$W^{1,2}_{d+1,\text{loc}} (\Omega_T)\cap
C(\bar{\Omega}_T)$-solutions we obtain
$u= v \in W^{1,2}_{p} (\Omega_T)$.
  This allows us to apply a priori estimates from Theorem 1.1 of
\cite{DKL} and along with \eqref{9.24.6}
proves \eqref{1.13.2}.

Finally, \eqref{2.28.1} 
follows from classical results 
(see, for instance, \cite{Kr85}, \cite{Li})
since $v$ satisfies \eqref{6.5.5}. 
 The theorem 
is proved.\qed

\mysection{Proof of Theorem \protect\ref{theorem 9.23.01}}
                                         \label{section 9.21.1}

As in Section \ref{section 12.13.2} we     
easily reduce proving Theorem \ref{theorem 9.23.01}
to proving the following.  

\begin{theorem}
                                   \label{theorem 10.5.01}
Suppose that $g\in C^{2}(\bR^{d})$ and
 Assumption \ref{assumption 9.23.1} is satisfied with
$\delta/2$ in place of $\delta$. Also assume that 
\eqref{9.10.2} 
holds  at all points of differentiability of  
$H (u,t,x)$ with respect to
$u$.
 Finally, assume that
estimates 
 \eqref{9.31.2} and \eqref{9.22.1}  
with a constant $N'$ and
\eqref{3.6.4} 
 hold
 for any $t,s\in \bR$, $x,y\in\bR^{d}$, and $u,v$.
Then the assertions of Theorem \ref{theorem 9.23.01}  hold true.
\end{theorem}

To prove Theorem \ref{theorem 10.5.01}
consider the equation
\begin{equation}
                                              \label{2.25.03}
\partial_t v+H_{K,h}[v]=0\quad \text{in}\quad
[0,T]\times\bR^{d}   
\end{equation}
 with terminal  condition
\begin{equation}
                                              \label{2.25.04}
v(T,x)=g(x)\quad\text{on}\quad \bR^{d}
 \end{equation}

In view of Picard's method  of  successive  approximations
 for any $h>0$ there 
exists a unique bounded solution $v=v_{ h}$ of
\eqref{2.25.03}--\eqref{2.25.04}.
Furthermore, $\partial_{t}v_{h}$ is bounded and continuous
with respect to $t$ for any $x$.  

We need a version of Lemma 4.2 of \cite{DK} for unbounded domains,
in which $Q^{o},\hat{Q}^{o},Q$ are generic
 objects described in Section \ref{section 5.3.1}
before assumption \eqref{9.25.1} was made.

\begin{lemma}
                                       \label{lemma 7.21.1}
Let $(a,b,c)(t,x)$ be
a bounded $\bR^{ m }\times\bR^{d}\times \bR$-valued
  function on
$\bR^{d+1}$  satisfying $a_k\ge 0$  and
$hb_{k}^{-}\leq  a_{k}$. Also let $h>0$ be small enough
for the arguments in the proof to go through.  
Let $v(t,x)$ be a bounded
function in $Q $ which is absolutely continuous
with respect to $t$
 on each open interval belonging to $Q^{o}_{|x}$
(if it is nonempty)
and for any $x\in  \Lambda_{\infty}^{ h}$
satisfying
$$
\partial_{t}v+Lv:=
\partial_{t}v+\sum_{k=1}^{m}a_{k}\Delta_{h,l_k}v
+\sum_{k=1}^{d}b_{k} \delta_{h,l_k}v-cv=-\eta
$$
(a.e.) on  $Q^{o}_{|x}$, where $\eta=\eta(t,x)$
 is a bounded function. Redefine $v$ if necessary
for $(t,x)\in\hat{Q}^{o}\setminus Q^{o}$ so that
$$
v(t,x)=\nlimsup_{s\uparrow t,(s,x)\in Q^{o}}v(s,x).
$$
Then  in $Q^{o}$ we have
$$
v\leq Te^{\bar{c}T}\sup_{Q^{o}}\eta_{+}+e^{\bar{c}T}
\sup_{Q\setminus Q^{o}}
v^{+},
$$
where $\bar{c}=\sup c^{-}$,
\end{lemma}

\begin{proof} First, as in \cite{DK} we reduce the general case
to the one where $c\geq0$. Then, by considering
$$
v(t,x)-(T-t) \sup_{Q^{o}}\eta^{+}-\sup_{Q\setminus Q^{o}}v^{+} ,
$$
we reduce the general case to the one with $\eta\leq0$
and $v\leq 0$ on $Q\setminus Q^{o}$.

Observe that for $\zeta(x)=\cosh |x|$ we have
$$
|D\zeta|+|D^{2}\zeta|\leq N'\zeta,
$$
where $N'$ depends only on $d$. It follows that for
a different $N'$, $h\in(0,1)$, and $k=1,...,m$
$$
|\delta_{h,l_{k}}\zeta|+|\Delta_{h,l_{k}}\zeta|\leq N'\zeta.
$$
Hence, the bounded function $w:=v\zeta^{-1}$ satisfies
$$
-\eta=\partial_{t}(w\zeta )+L(w\zeta )
=\zeta\partial_{t}w+\zeta\sum_{k=1}^{m}a_{k}\Delta_{h,l_k}w
$$
$$
+\zeta\sum_{k=1}^{m}a_{k}[c_{-k}\delta_{h,-l_{k}}w
+c_{ k}\delta_{h, l_{k}}w]
+\zeta\sum_{k=1}^{d}\bar{b}_{k}\delta_{h,l_{k}}w
+\zeta \bar{c}w
$$
where
$c_{\pm k}=\zeta^{-1}\delta_{h,\pm l_{k}}\zeta$,
$\bar{b}_{k}=b_{k}\zeta^{-1}T_{h,l_{k}}\zeta$,
$$
\bar{c}=-c+\zeta^{-1}\sum_{k=1}^{m}\Delta_{h,l_{k}}\zeta
+\zeta^{-1}\sum_{k=1}^{d}b_{k}\delta_{h,l_{k}}\zeta.
$$
It follows that for any constant $\lambda>0$ we have 
$$
\partial_{t}(we^{\lambda(T-t)})
+\sum_{k=1}^{m}a_{k}\Delta_{k}(we^{\lambda(T-t)})
+ \sum_{k=1}^{d}\bar{b}_{k}\delta_{h,l_{k}}(we^{\lambda(T-t)})
$$
\begin{equation}
                                                    \label{9.9.1}
+ \sum_{k=1}^{m}a_{k}[c_{-k}\delta_{h,-l_{k}}
+c_{ k}\delta_{h, l_{k}} ](we^{\lambda(T-t)})
+ ( \bar{c}-\lambda)(we^{\lambda(T-t)})\geq0.
\end{equation}
For $\lambda$ sufficiently large and $h$ sufficiently small
we have $\bar{c}-\lambda\leq0$  and the coefficients
in \eqref{9.9.1} satisfy other conditions of 
 Lemma 4.2 of \cite{DK} {\color{black}
which guarantee that the finite-difference operator
involved in the left-hand side of \eqref{9.9.1}
obeys the maximum principle, that is
$$
-h\bar{b}_{k}+h2a_{k} |c_{ k}|\leq 2a_{k}
$$
for all $k$ which is true if $h$ is sufficiently small.
This} allows us to conclude that
for any $R\in(0,\infty)$ on 
$  Q^{o}\cap\big[(0,T)\times B_{R}\big]$ we have
$$
w(t,x)e^{\lambda(T-t)}\leq\sup\{w^{+}(s,x)
e^{\lambda(T-s)}:(s,x)\in Q, |x|\geq R\}.
$$
Here the right-hand side goes to zero as $R\to\infty$
since $|w|=|v|\zeta^{-1}$ and $v$ is bounded. Hence
$w\leq0$ and this proves the lemma.
\end{proof}

\begin{corollary}
                                        \label{corollary 9.10.1}
There exists a constant $N$ depending only on
  $d$  and $K_{0}$ such that for all
sufficiently small $h$ we have
\begin{equation}
                                           \label{9.11.5}
|v_{h}|\leq Ne^{NT}(\bar{H}+K+\sup|g|).
\end{equation}
\end{corollary}
This corollary is obtained from
  Lemma \ref{lemma 7.21.1} 
by repeating the first part of the proof of 
Lemma \ref{lemma 9.13.1}.

\begin{lemma}
                                         \label{lemma 9.14.2}
There exists a constant $N$ depending only on
  $d$, $\delta$, $T$, and $K_{0}$ such that for all
sufficiently small $h$ we have
\begin{equation}
                                              \label{eq10.12}
  |\partial_t v_{ h}|
\leq N(\bar{H}+K+\|g\|_{C^{ 2}(\bR^{d})}+
\sup_{(0,T)\times\bR^{d}}\sum_{k=1}^{m}|
\delta_{h,l_{k}}v_{h}|),
\end{equation}
\begin{equation}
                                              \label{9.11.3}
  \sum_{k=1}^{m}|\Delta_{h,l_{k}} v_{ h}|
\leq N(\bar{H}+K+\|g\|_{C^{ 2}(\bR^{d})}+
\sup_{(0,T)\times\bR^{d}}\sum_{k=1}^{m}|
\delta_{h,l_{k}}v_{h}|)
\end{equation}
on $(0,T)\times\bR^{d}$. 
\end{lemma}

\begin{proof} One proves \eqref{eq10.12} in the same way as
\eqref{9.13.2} with the only difference that
instead of Lemma 4.2 of \cite{DK} one uses
Lemma \ref{lemma 7.21.1}.

In case of  \eqref{9.11.3} we add to \eqref{9.16.7}
the fact that the left-hand side of \eqref{9.16.7}
is nonpositive outside 
$$
Q^{o} :=\{(t,x)\in(0,T)\times\Lambda^{h}_{\infty}:
(\hat\delta/2)\sum_{k=1}^{m}
|\Delta_{h,l_{k}}v_{h}(t,x)|> 
$$
$$
>\bar{ H}+K+K_{0}
\big(|v_{h}(t,x)|+M_{h}(t,x)  
\big)\},
$$
 Hence, for any
$r\in\{1,...,m\}$ on $Q^{o}$ there exist
functions $a_{k} $
satisfying $\hat{\delta}/2\leq a_{k}\leq 2\hat{\delta}^{-1}$
 such that
on every $x$-section of $Q^{o}$ (a.e.)  we have
$$
\partial_t (\Delta_{h,l_{r}}v_h) + a_k\Delta_{h,l_k} 
(\Delta_{h,l_{r}}v_h) \leq0.
$$
It follows by Lemma \ref{lemma 7.21.1} that in $Q^{o}$
$$
(\Delta_{h,l_{r}}v_h)^{-} 
\leq\sup_{ (0,T] \times \Lambda^{h}_{\infty}
  \setminus  Q^{o}}
(\Delta_{h,l_{r}}v_h)^{-}.
$$
Now the continuity of $\Delta_{h,l_{r}}v_h$ with respect to $t$
and  the definition of $Q^{o}$
show  that
 $(\Delta_{h,l_{r}}v_h)^{-} $ is dominated by the right-hand side of
\eqref{9.11.3}.
Then equation \eqref{9.11.1} combined with estimates
\eqref{9.11.2}, \eqref{eq10.12}, and \eqref{9.11.5}
allow us to conclude that also 
$(\Delta_{h,l_{r}}v_h)^{+} $ is dominated by the right-hand side of
\eqref{9.11.3}. This proves the lemma. 
\end{proof}

Our next step is to exclude $|\delta_{h,l_{k}}v_{h}|$ from the 
right-hand side of \eqref{eq10.12} and \eqref{9.11.3}
by using interpolation, that is by using \eqref{9.24.7},
which for $w(i) =v_{h}(t,x+ihl_{k})$,
where $(t,x)\in(0,T)\times\bR^{d}
$,    $h<1$, and integer $r\geq 2$ yields that 
$$
|\delta_{h,l_{k}}v_{h}(t,x)|
\leq \frac{1}{2}rh\max_{|i|\leq r}|
\Delta_{h,l_{k}}v_{h}(t,x+ihl_{k})|+\frac{4}{rh}
\max_{|i|\leq r}|
 v_{h}(t,x+ihl_{k})|.
$$
In light of the arbitrariness of $r\geq2$ and
 \eqref{9.11.5}
 and \eqref{9.11.3} we conclude that
for any $\varepsilon\geq 2h$
$$
|\delta_{h,l_{k}}v_{h}|\leq
N\varepsilon^{-1}(\bar{H}+K+\sup|g|)
$$
$$
+N\varepsilon(\bar{H}+K+\|g\|_{C^{ 2}(\bR^{d})}+
\sup_{(0,T)\times\bR^{d}}\sum_{k=1}^{m}|
\delta_{h,l_{k}}v_{h}|).
$$

 It follows that for all sufficiently
small $h$ we have 
\begin{equation}
                                               \label{9.11.6}
\sup_{(0,T)\times\bR^{d}}\sum_{k=1}^{m}|
\delta_{h,l_{k}}v_{h}|\leq
N(\bar{H}+K+\|g\|_{C^{ 2}(\bR^{d})}),
\end{equation}
\begin{equation}
                                               \label{9.11.07}
\sup_{(0,T)\times\bR^{d}}\big(|v_{h}|+
|\partial_{t}v_{h}|+\sum_{k=1}^{m}|
\Delta_{h,l_{k}}v_{h}|\big)\leq
N(\bar{H}+K+\|g\|_{C^{ 2}(\bR^{d})}).
\end{equation}

{\color{black}Observe that, in contrast with \eqref{9.18.9},
  \eqref{9.11.07} yields a global estimate
of $\Delta_{h,l_{k}}v_{h}$. This allows us to repeat the
proof of Lemma \ref{lem11.24} without excluding
$u'_{0}$ from $\cH_{K}$ and in place of \eqref{6.3.7}
obtain
\begin{equation}
                                          \label{6.8.1}
\partial_{t}w_{h}+a_{hk}\Delta_{h,l_{k}}w_{h} 
+b_{hk}\delta_{h,e_{k}}w_{h}+c_{hk} w_{h}
+f_{h}h =0
\end{equation}
in $(0,T)\times\bR^{d}$,
which implies the following.
}
 
\begin{corollary}
                                        \label{corollary 9.18.1}
There is a constant $M$, 
which may depend  on $N'$, such that for
all  $h>0$, $t \in (0,T]$, and $x,y \in \bR^{d}$, we have
$$
|v_{h}(t,x)-v_{h}(t,y)|\leq
M(|x-y|+h).
$$
\end{corollary}

After that one finishes the proof of Theorem \ref{theorem 9.23.01}
in the same way as Theorem \ref{theorem 10.5.1} is proved,
{\color{black} of course, dropping the part of the proof dealing
with the fixed point argument}.

\mysection{Proof of Theorem \protect\ref{theorem 9.15.1}}
                                       \label{section 9.21.3}

By the maximum principle $v_{K}$ decreases as $K$ increases.
Estimate \eqref{2.28.1} guarantees that $v_{K}$ converges uniformly
to a function $v\in C(\bar{\Omega}_{T})$. To prove that
$v$ is an $L_{d+1}$-viscosity solution we need the following,
in which 
$$ 
C_{r}=(0,r^{2})\times
B_{r},\quad C_{r}(t,x)=(t,x)+C_{r}. 
$$

\begin{lemma}
                                           \label{lemma 9.20.1}
There is a constant $N$ depending only on $d$,
$\delta$, and the Lipschitz constant of $H$ with respect to
$(u'_{1},...,u'_{d})$ such that
for any $r\in(0,1]$ and $C_{r}(t,x)$ satisfying $C_{r}(t,x)\subset\Omega_{T}$  and
$\phi\in W^{1,2}_{d+1}(C_{r}(t,x))$ we have on $C_{r}(t,x)$ that
\begin{equation}
                                             \label{9.20.1}
v\leq \phi+Nr^{d/(d+1)}\|(\partial_{t}\phi+
H[\phi])^{+}\|_{L_{d+1}(C_{r}(t,x))}
+\max_{\partial'C_{r}(t,x)}(v-\phi)^{+} .
\end{equation}
\begin{equation}
                                             \label{9.20.2}
v\geq \phi-Nr^{d/(d+1)}\|(\partial_{t}\phi+
H[\phi])^{-}\|_{L_{d+1}(C_{r}(t,x))}
-\max_{\partial'C_{r}(t,x)}(v-\phi)^{-} .
\end{equation}
\end{lemma}

\begin{proof} Observe that
$$
-\partial_{t}\phi-\max(H[\phi],P[\phi]-K)=
-\partial_{t}\phi-\max(H[\phi],P[\phi]-K)
$$
$$
+\partial_{t}v_{K}+
\max(H[v_{K}],P[v_{K}]-K)
$$
$$
=
\partial_{t}(v_{K}-\phi)+a_{ij}D_{ij}(v_{K}-\phi)
+b_{i}D_{i}(v_{K}-\phi)-c(v_{K}-\phi),
$$
where $a=(a_{ij})$ is a $d\times d$ symmetric matrix-valued
function whose eigenvalues are in $[\check{\delta} ,\check{\delta}^{-1}]$,
$b_{i}$ are bounded functions, and $c\geq0$.
It follows by Lemma 2.1 and Remark 1.1 of \cite{Kr86} with
$$
u=v_{K}-\phi-\max_{\partial'C_{r}(t,x)}(v_{K}-\phi) ^{+}
$$
that for $r\in(0,1]$
$$
v_{K}\leq \phi+\max_{\partial'C_{r}(t,x)}(v_{K}-\phi)^{+}
$$
\begin{equation}
                                             \label{9.20.3}
+Nr^{d/(d+1)}\|(\partial_{t}\phi+
\max(H[\phi],P[\phi]-K))^{+}\|_{L_{d+1}(C_{r}(t,x))},
\end{equation}
where the constant $N$ is of the type described
in the statement of the present lemma.
 We obtain \eqref{9.20.1} from \eqref{9.20.3}
by letting $K\to\infty$.
In the same way \eqref{9.20.2} is established.
The lemma is proved.
\end{proof}

Now we can prove that $v$ is  an $L_{d+1}$-viscosity solution.
Let $(t_{0},x_{0})\in\Omega_{T}$ and $\phi\in W^{1,2}_{d+1,loc}
(\Omega_{T})$ be such that $v-\phi$ attains a local
maximum at $(t_{0},x_{0})$ and $v(t_{0},x_{0})=\phi(t_{0},x_{0})$.
 Then for   $\varepsilon
>0$ and all small $r>0$ for
$$
\phi_{\varepsilon,r}(t,x)=\phi (t,x)+\varepsilon(
|x-x_{0}|^{2}+t-t_{0}- r^{2})
$$
 we have that 
$$
\max_{\partial'C_{r}(t_{0},x_{0})}(v -\phi_{\varepsilon,r})^{+}
=0.
$$
Hence, by Lemma \ref{lemma 9.20.1}
$$
   \varepsilon r^{2}= 
(v -\phi_{\varepsilon,r})(t_{0},x_{0})
\leq Nr^{d/(d+1)}\|(\partial_{t}\phi_{\varepsilon,r}+
H[\phi_{\varepsilon,r}])^{+}\|_{L_{d+1}(C_{r}(t_{0},x_{0}))},
$$
$$
Nr^{-(d+2)}\|(\partial_{t}\phi_{\varepsilon,r}+
H[\phi_{\varepsilon,r}])^{+}\|^{d+1}_{L_{d+1}(C_{r}(t_{0},x_{0}))}
\geq \varepsilon^{d+1}.
$$
By letting $r\downarrow0$ and using the continuity
of $H(u,t,x)$ in $u'_{0}$, which is assumed to be uniform
with respect to other variables, we obtain
\begin{equation}
                                        \label{6.8.4}
N\lim _{ r\downarrow0}\esssup_{C_{r}(t_{0},x_{0}) }
(\partial_{t}\phi_{\varepsilon }+
H[\phi_{\varepsilon }])\geq \varepsilon.
\end{equation}
where $\phi_{\varepsilon }=\phi+\varepsilon
(|x-x_{0}|^{2}+t-t_{0})$. {\color{black}Finally, observe that
 $v$ is continuous by construction,
$\phi$ is locally continuous by embedding theorems,
and $H(u,t,x)$ is  continuous with respect to
$u$ uniformly with respect to $(t,x)$ by assumption.}
Then letting $\varepsilon\downarrow0$ in \eqref{6.8.4}
proves that
$v$ is an $L_{d+1}$-viscosity subsolution.
The fact that it is also an $L_{d+1}$-viscosity supersolution
is proved similarly on the basis of \eqref{9.20.2}.

Finally, we prove that $v$ is the maximal continuous
$L_{d+1}$-viscosity subsolution. Let $u$ be an 
$L_{d+1}$-viscosity subsolution of \eqref{7.29.1} of class 
$C(\bar{\Omega}_{T})$.
Then, as is easy to see, for any $K\geq0$,
$u-v_{K}$   is an $L_{d+1}$-viscosity subsolution of  
$$
\partial_{t}w+F[w]= -h_{K},
$$
where $F[w]:=H[w+v_{K}]-H[v_{K}]$, so that $F[0]=0$, and
$ 
h_{K}(t,x)= H[v_{K}]  
$.
Since $h_{K}\leq0$, we conclude by Proposition 2.6
of \cite{CKS00} that, if, additionally, $u=g$ on $\partial\Omega_{T}$,
then $u-v_{K}\leq0$ in $\Omega_{T}$. Now it only remains to
let $K\to\infty$. The theorem is proved.\qed 

\begin{remark}
As  follows from \cite{CKS00} continuous
$L_{d+1}$-viscosity subsolutions $u$ of \eqref{7.29.1} satisfy
\eqref{9.20.1} with $u$ in place of $v $ for any
$\phi\in W^{1,2}_{d+1}(C_{r}(t,x))$ whenever
 $r\in(0,1]$ and
$C_{r}(t,x)\subset\Omega_{T}$. Therefore, this relation
can be taken as an equivalent definition
of what $L_{d+1}$-viscosity subsolutions are.
A nice feature of \eqref{7.29.1} is that it is 
satisfied for any
$\phi\in W^{1,2}_{d+1}(C_{r}(t,x))$ iff it is satisfied
for any $\phi\in C^{1,2}(\bar C_{r}(t,x))$.
\end{remark}

\medskip

{\color{black}  
{\bf Acknowledgement} The author is sincerely grateful
to many comments of the three referees, which  certainly helped
improve the presentation of the paper.}

\end{document}